\documentclass[10pt]{article}
\def\date{1 December 2012 Last Revised 22 August 2017}
\usepackage{amsmath, amssymb, amsfonts, amsthm}

\newtheorem{proposition}{proposition}[section]
\newtheorem{lemma}[proposition]{Lemma}
\newtheorem{corollary}[proposition]{Corollary}
\newtheorem{theorem}[proposition]{Theorem}
\newtheorem{claim}[proposition]{Claim}

\newtheorem{prop}[proposition]{Proposition}
\newtheorem{lem}[proposition]{Lemma}
\newtheorem{cor}[proposition]{Corollary}
\newtheorem{thm}[proposition]{Theorem}

\newtheorem{subclaim}{Subclaim}
{
\newtheorem{definition}[proposition]{Definition}

}

\newcommand{\mcal}{\mathcal}

\newcommand{\C}{\mcal{C}}

\newcommand{\F}{\mcal{F}}

\newcommand{\R}{\mcal{R}}

\begin{document}
\font\smallrm=cmr8





\baselineskip=12pt
\phantom{a}\vskip .25in
\centerline{{\bf  $3$-List-Coloring Graphs of Girth at least Five on Surfaces}}
\vskip.4in
\centerline{{\bf Luke Postle}
\footnote{\texttt{Partially supported by NSERC under Discovery Grant No. 2014-06162, the Ontario Early Researcher Awards program and the Canada Research Chairs program. Email: lpostle@uwaterloo.ca}}} 
\smallskip
\centerline{Department of Combinatorics and Optimization}
\centerline{University of Waterloo}
\centerline{Waterloo, ON}
\centerline{Canada, N2L 3G1}

\vskip 1in \centerline{\bf ABSTRACT}
\bigskip

{
\parshape=1.0truein 5.5truein
\noindent

Gr\"otzsch proved that every triangle-free planar graph is $3$-colorable. Thomassen proved that every planar graph of girth at least five is $3$-choosable. As for other surfaces, Thomassen proved that there are only finitely many $4$-critical graphs of girth at least five embeddable in any fixed surface. This implies a linear-time algorithm for deciding $3$-colorablity for graphs of girth at least five on any fixed surface. Dvo\v{r}\'ak, Kr\'al' and Thomas strengthened Thomassen's result by proving that the number of vertices in a $4$-critical graph of girth at least five is linear in its genus. They used this result to prove Havel's conjecture that a planar graph whose triangles are pairwise far enough apart is $3$-colorable. As for list-coloring, Dvo\v{r}\'ak proved that a planar graph whose cycles of size at most four are pairwise far enough part is $3$-choosable.

In this article, we generalize these results. First we prove a linear isoperimetric bound for $3$-list-coloring graphs of girth at least five. Many new results then follow from the theory of hyperbolic families of graphs developed by Postle and Thomas. In particular, it follows that there are only finitely many $4$-list-critical graphs of girth at least five on any fixed surface, and that in fact the number of vertices of a $4$-list-critical graph is linear in its genus. This provides independent proofs of the above results while generalizing Dvo\v{r}\'ak's result to graphs on surfaces that have large edge-width and yields a similar result showing that a graph of girth at least five with crossings pairwise far apart is $3$-choosable. Finally, we generalize to surfaces Thomassen's result that every planar graph of girth at least five has exponentially many distinct $3$-list-colorings. Specifically, we show that every graph of girth at least five that has a $3$-list-coloring has $2^{\Omega(n)-O(g)}$ distinct $3$-list-colorings.

}

\vfill \baselineskip 11pt \noindent \date.
\vfil\eject
\baselineskip 18pt

\section{Introduction}
All graphs considered in this paper are simple and finite. Graph coloring is an important area of study in graph theory. Recall that a \emph{coloring} of a graph $G$ is an assignment of colors to vertices such that no two adjacent vertices receive the same color. A \emph{$k$-coloring} is a coloring that uses at most $k$ colors while a graph $G$ is \emph{$k$-colorable} if there exists a $k$-coloring of $G$. 

List coloring, also known as choosability, generalizes the concept of coloring and was introduced by Vizing~\cite{Vizing} and independently by Erd\H{o}s et al.~\cite{Erdos}. 

\begin{definition}
A \emph{list assignment} of $G$ is a function $L$ that assigns to each vertex $v\in V(G)$ a list $L(v)$ of colors. An \emph{$L$-coloring} of $G$ is a function $\phi: V(G)\rightarrow \bigcup_{v\in V(G)} L(v)$ such that $\phi(v)\in L(v)$ for every $v\in V(G)$ and $\phi(u)\ne \phi(v)$ for every pair of adjacent vertices $u,v$ in $G$. We say $G$ is \emph{$L$-colorable} if $G$ has an $L$-coloring. 

A \emph{$k$-list-assignment} is a list-assignment $L$ such that $|L(v)|\ge k$ for all $v\in V(G)$. A graph $G$ is \emph{$k$-list-colorable} or \emph{$k$-choosable} if $G$ is $L$-colorable for every $k$-list-assignment $L$.
\end{definition}

A well-known result of Gr\"otzsch~\cite{Grotzsch} states that every triangle-free planar graph is $3$-colorable. This theorem does not extend to list-coloring as Voigt~\cite{Voigt} constructed a triangle-free planar graph that is not 3-choosable. However, Thomassen~\cite{ThomList} proved that every planar graph of girth at least $5$ is $3$-choosable where \emph{girth} is the length of the smallest cycle.

A natural extension of such results is to graphs on surfaces. For terms related to graphs embedded in surfaces, we refer to \cite{MoharThom}. Since not every graph is $3$-colorable and coloring is a \emph{monotone} property, that is, $\chi(H)\le \chi(G)$ for every $H \subseteq G$, it is natural to consider the minimal non-colorable graphs. Similarly, choosability is a monotone property. Hence the following definitions.

\begin{definition}
A graph $G$ is \emph{$k$-critical} if $G$ is not $(k-1)$-colorable but every proper subgraph of $G$ is. A graph $G$ is \emph{$k$-list-critical} if there exists a $k$-list-assignment $L$ for $V(G)$ such that $G$ is not $L$-colorable but every proper subgraph of $G$ is $L$-colorable.
\end{definition}

Thomassen~\cite{ThomRegular} proved that there are only finitely many $4$-critical graphs of girth at least five embeddable in a fixed surface. Dvo\v{r}\'ak, Kr\'al' and Thomas~\cite{DvoKraTho3} strengthened this result by proving that the number of vertices in a $4$-critical graph of girth at least five is linear in its genus. One of the main results of this paper is to generalize these results to list-coloring as follows.

\begin{theorem}\label{LinearGenus}
There exists a constant $c$ such that if $G$ is a $4$-list-critical of girth at least five embedded in a surface of genus $g$, then $|V(G)|\le cg$.
\end{theorem} 

As an immediate corollary of Theorem~\ref{LinearGenus}, we have the following.

\begin{corollary}\label{FinMany}
For every surface $S$, there exist only finitely many $4$-list-critical graphs of girth at least five embeddable in $S$.
\end{corollary}

Using a result of Eppstein~\cite{Eppstein} that testing for a fixed subgraph on a fixed surface can be done in linear time, we obtain the following. 

\begin{corollary}\label{LinearAlg}
For a fixed surface $S$, testing if a graph $G$ of girth at least five embedded in $S$ is $3$-choosable can be done in linear time. 
\end{corollary}

Moreover, Postle and Thomas~\cite{PostleThomas} deduced from Theorem~\ref{LinearGenus} that, for a fixed surface $S$, testing if a graph of girth at least five embedded in $S$ can be colored from a given $3$-list-assignment can be done in linear time, which is a theorem of Dvo\v{r}\'ak and Kawarabayashi~\cite{DvoKaw2}.

Another rather immediate corollary of Theorem~\ref{LinearGenus} is that locally planar graphs of girth at least five are $3$-choosable. More precisely, recall that the \emph{edge-width} of an embedded graph is the length of its shortest non-contractible cycle. The corollary then is as follows.

\begin{corollary}\label{LocPlanar}
For every surface $S$, there exists a constant $c(S)$ such that if $G$ is a graph of girth at least five and the edge-width of $G$ is at least $c(S)$, then $G$ is $3$-choosable.
\end{corollary}

Corollary~\ref{LocPlanar} follows from Corollary~\ref{FinMany} by letting $c(S)$ be strictly larger than the maximum number of vertices in a $4$-list-critical graph embeddable in $S$. Hence it follows from Theorem~\ref{LinearGenus} that $c(S) = O(g)$ where $g$ is the genus of $S$. Postle and Thomas~\cite{PostleThomas} improved this further by showing that $c(S) = O(\log g)$ which is best possible since there exists expander graphs with girth $\Omega(\log g)$ and high chromatic number.

For ordinary coloring, Thomassen~\cite{ThomRegular} derived similar consequences about locally planar graphs and algorithms from his theorem that there are only finitely many $4$-critical graphs of girth at least five embeddable in a fixed surface. Indeed, A key approach developed by Thomassen to prove these kinds of results is to consider a subgraph $H$ of a graph $G$. We say a coloring of $H$ \emph{extends} to a coloring of $G$ if the two colorings agree on all vertices of $H$. The key then is to prove that there exists a subgraph $G'$ of $G$, whose size depends only on the size of $H$, such that any coloring of $H$ extends to $G$ if and only if that coloring extends to $G'$. 

In recent years, researchers have realized that proving a linear bound on such subgraphs, namely $|V(G')|=O(|V(H)|)$, is not only desirable in the sense that this is usually best possible but also has many striking consequences. Indeed, a linear bound for such precolored subgraphs for ordinary $3$-coloring is the key to the work of Dvo\v{r}\'ak, Kr\'al' and Thomas~\cite{DvoKraTho} alluded to above (see~\cite{DvoKraTho2} and~\cite{DvoKraTho3}). 

Of central importance is the case when $G$ is a plane graph and $H$ is its outer cycle.  Dvo\v{r}\'ak and Kawarabayashi~\cite{DvoKaw} have proven a linear bound for $3$-list-coloring graphs of girth at least five when $G$ is planar and $H$ has one component. To be more precise, we need the following definition.

\begin{definition}
A graph $G$ is \emph{$C$-critical} (with respect to some list assignment $L$) if for every proper subgraph $H\subset G$ such that $C\subseteq H$, there exists an $L$-coloring of $C$ that extends to an $L$-coloring of $H$, but not to an $L$-coloring of $G$. 
\end{definition}

Dvo\v{r}\'ak and Kawarabayashi~\cite{DvoKaw} proved that if $G$ is a plane graph of girth at least five that is $C$-critical with respect to some $3$-list-assignment $L$ where $C$ is its outer cycle, then $|V(G)|\le \frac{37}{3}|V(C)|$. This linear bound implies that the family of $4$-list-critical graphs of girth at least five is ``hyperbolic", a precise definition of which can be found in Section~\ref{Exp}, but informally says that there exists a constant $K$ such that for every graph in the family, the number of vertices inside a disk is at most $K$ times the number of vertices on the boundary.

The theory of hyperbolic families, developed by Postle and Thomas~\cite{PostleThomas}, then implies a number of striking consequences from this fact, namely Corollary~\ref{LocPlanar} with $c(S)=O(\log g)$ and - following Dvo\v{r}\'ak and Kawarabayashi~\cite{DvoKaw2} - a linear-time algorithm for $3$-list-coloring a graph of girth at least five on a fixed surface. 

The main result of this paper is to prove that the family of $4$-list-critical graphs is in fact ``strongly hyperbolic" (the precise definition of which can also be found in Section~\ref{Exp} but can be thought of as requiring a linear bound not only for disks but annuli). This family being strongly hyperbolic is implied by the following theorem (see~\cite{PostleThomas} for the details of this implication), which an extension of Dvorak and Kawaravayashi's result to two precolored cycles.

\begin{theorem}\label{Cylinder}
Let $G$ be a plane graph of girth at least five, let $L$ be a $3$-list assignment of $G$ and let $C_1\ne C_2$ be facial cycles of $G$. If $G$ is $C_1\cup C_2$-critical with respect to $L$, then $|V(G)|\le 177(|V(C_1)|+|V(C_2)|)$.
\end{theorem}

Hence we obtain the following corollary.

\begin{corollary}\label{StronglyHyper}
The family of $4$-list-critical graphs of girth at least five is strongly hyperbolic.
\end{corollary}

Postle and Thomas~\cite{PostleThomas} proved that if $\F$ is a strongly hyperbolic family of graphs, then there exists $c_{\F}$ such that for every graph $G\in \F$, $|V(G)|\le c_{\F}g$ where $g$ is the genus of $G$. That theorem combined with Corollary~\ref{StronglyHyper} implies Theorem~\ref{LinearGenus}. Moreover the proof of Theorem~\ref{Cylinder} does not directly rely on the theorems of Thomassen, or Dvo\v{r}\'ak, Kr\'al' and Thomas, or Dvo\v{r}\'ak and Kawarabayashi and hence provides independent proofs of these results (though in the latter case with a larger constant). 

Thus the majority of this paper is devoted to proving Theorem~\ref{Cylinder}. However, the consequences of Corollary~\ref{StronglyHyper} go beyond even Theorem~\ref{LinearGenus} as explained in full detail in~\cite{PostleThomas}. In the next few subsections, we highlight some of these further applications to related problems about $3$-list-coloring graphs of girth at least five on surfaces.

\subsection{A Linear Bound for Precolored Subgraphs}

Postle and Thomas~\cite{PostleThomas} showed that Theorem~\ref{Cylinder} is equivalent to the strong hyperbolicity of a slightly more general family of graphs as follows. We say $(G,H)$ is a \emph{graph with boundary} if $H$ is a subgraph of the graph $G$. We can then extend the notion of hyperbolicity and strong hyperbolicity to families of graph with boundaries by requiring that the disk (or annuli) not contain any edge or vertex of $H$ in its interior. With this terminology, Theorem~\ref{Cylinder} is equivalent to the following.

\begin{theorem}\label{CylinderBoundary}
The family of graphs with boundary $(G,H)$ where $G$ is a graph of girth at least five such that $G$ is $H$-critical with respect to some $3$-list-assignment is strongly hyperbolic. 
\end{theorem}

Postle and Thomas then showed that if $\F$ is a strongly hyperbolic family with boundary, then there exists $c_{\F}$ such that for every $(G,H)\in \F$, $|V(G)|=c_{\F}(g+|V(H)|)$ where $g$ is the genus of $G$. Combining that theorem with Theorem~\ref{CylinderBoundary} gives the following corollary.

\begin{corollary}\label{LinearCriticalBound}
There exists a constant $c$ such that if $G$ is $H$-critical with respect to some $3$-list-assignment $L$ for some subgraph $H$ of $G$, then $|V(G)|\le c(g+|V(H)|)$ where $g$ is the genus of $G$.
\end{corollary} 

Corollary~\ref{LinearCriticalBound} is a far-reaching generalization of Theorem~\ref{Cylinder} (where $g=0$ and $H$ has at most 2 components), though the constant is much larger than $177$ (around $5,000,000$ if one does the calculations in~\cite{PostleThomas} using 177 and the constant of $37/3$ from Dvo\v{r}\'ak and Kawarabayashi for one cycle). 

\subsection{Precolored Cycles and Crossings Far Apart}

Using Theorem~\ref{CylinderBoundary}, combined with a structure theorem for strongly hyperbolic families of graphs and the fact that every graph of girth at least five embeddable in the plane with at most one crossing is $3$-choosable (which can derived from Thomassen's original proof of $3$-choosability of graphs of girth at least five), Postle and Thomas~\cite{PostleThomas} proved the following.

\begin{theorem}\label{Precolored2}
There exists $D>0$ such that the following holds: Let $G$ be a graph of girth at least five $2$-cell embedded in a surface $S$ of genus $g$ such that the edge-width of $G$ is $\Omega(g)$ and let $L$ be a $3$-list-assignment for $G$. If $X\subset V(G)$ such that $d(u,v)\ge D$ for all $u\ne v\in X$, then every $L$-coloring of $X$ extends to an $L$-coloring of $G$.
\end{theorem}

Indeed, Postle and Thomas proved a stronger version of Theorem~\ref{Precolored2} when there are precolored cycles of length at most four far enough apart. 

\begin{theorem}\label{PrecoloredCycles}
There exists $D>0$ such that the following holds: Let $G$ be a graph $2$-cell embedded in a surface $S$ of genus $g$ such that the edge-width of $G$ is $\Omega(g)$ and let $L$ be a $3$-list-assignment for $G$. Let $\C$ be the set of cycles of $G$ of length at most four. If $d(C_i,C_j)\ge D$ for all $C_i\ne C_j \in \C$ and each $C_i\in \C$ is homotopically trivial, then if $\phi$ is an $L$-coloring of the cycles in $\C$, then $\phi$ extends to an $L$-coloring of $G$.
\end{theorem}

When $S$ is the plane, this was proved by Dvo\v{r}\'ak~\cite{Dvo} and hence Theorem~\ref{PrecoloredCycles} provides an independent proof of his result. His result is actually an analogue of Havel's conjecture for list-coloring. 

Havel's conjecture~\cite{Hav1, Hav2} states that there exists $d>0$ such that if all the triangles in a planar graph $G$ are pairwise distance at least $d$ apart, then $G$ is $3$-colorable. Dvo\v{r}\'ak, Kr\'al' and Thomas~\cite{DvoKraTho5} proved Havel's conjecture (see also~\cite{DvoKraTho}). An essential ingredient of their proof is proving that the family of $4$-critical graphs of girth at least five is strongly hyperbolic, for which Theorem~\ref{Cylinder} provides an independent (and arguably shorter) proof. Dvo\v{r}\'ak's result is a natural analogue of Havel's conjecture for list-coloring as there exist triangle-free planar graphs which are not $3$-choosable. 

Postle and Thomas~\cite{PostleThomas} also deduced the following theorem from Theorem~\ref{CylinderBoundary}.

\begin{theorem}\label{CrossingSurface}
There exists $D>0$ such that the following holds: Let $G$ be a graph of girth at least five drawn in a surface $S$ of genus $g$ with a set of crossings $X$ and $L$ be a $3$-list-assignment for $G$. Let $G_X$ be the graph obtained by adding a vertex $v_x$ at every crossing $x\in X$. If the edge-width of $G_X$ is $\Omega(g)$ and $d(v_x,v_{x'})\ge D$ for all $v_x \ne v_{x'} \in V(G_X)\setminus V(G)$, then $G$ is $L$-colorable.
\end{theorem}

For ordinary $3$-coloring, the analogues of Theorems~\ref{Precolored2},~\ref{PrecoloredCycles}, and~\ref{CrossingSurface} may be derived from Dvo\v{r}\'ak, Kr\'al' and Thomas' work~\cite{DvoKraTho5}.

\subsection{Exponentially Many List Colorings}

For ordinary coloring, Thomassen gave a suprisingly short proof that every planar graph of girth at least five has at least $2^{\frac{|V(G)|}{9}}$ distinct $3$-colorings by using the edge-density of planar graphs of girth at last five and the fact that such a graph has at least one $3$-coloring by Gr\"otzsch's Theorem. Furthermore, Thomassen's work easily implies that for every surface $S$, there exists $c_S$ such that if a graph $G$ of girth at least five embedded in $S$ has at least one $3$-coloring, then it has at least $c_S2^{\frac{|V(G)|}{9}}$ distinct $3$-colorings. 

As for list-coloring, Thomassen~\cite{ThomExp} in a deeper result proved that a planar graph $G$ of girth at least five has at least $2^{\frac{|V(G)|}{10000}}$ distinct $L$-colorings for any $3$-list-assignment $L$. Once again using Theorem~\ref{CylinderBoundary} and a structure theorem for strongly hyperbolic families of graphs, Postle and Thomas~\cite{PostleThomas} extended this to all surfaces as follows.

\begin{theorem}\label{ExpSurfacetheorem}
There exists $\epsilon > 0$ such that: For every surface $S$ there exists a constant $c_S>0$ such that following holds: Let $G$ be a graph of girth at least five embedded in $\Sigma$ and $L$ a $3$-list-assignment for $G$. If $G$ has an $L$-coloring, then $G$ has at least $c_S 2^{\epsilon|V(G)|}$ $L$-colorings of $G$.
\end{theorem}

Indeed, they prove a stronger version of Theorem~\ref{ExpSurfacetheorem} about extending a precoloring of a subset of the vertices as follows. Note if $G$ is a graph and $R\subseteq V(G)$, we let $G[R]$ denote the subgraph of $G$ induced by $R$.

\begin{theorem}\label{ExpSurface2}
There exist constants $\epsilon,\alpha>0$ such that following holds: Let $G$ be a graph of girth at least five embedded in a surface $S$ of genus $g$, $R\subseteq V(G)$ and $L$ a $3$-list-assignment for $G$. If $\phi$ is an $L$-coloring of $G[R]$ such that $\phi$ extends to an $L$-coloring of $G$, then $\phi$ extends to at least $2^{\epsilon(|V(G)|-\alpha(g+|R|))}$ $L$-colorings of $G$.
\end{theorem}

To prove Theorem~\ref{ExpSurface2}, they showed that it suffices to prove that the family of graphs of girth at least five with boundary which are `critical' with respect to not having exponentially many extensions is strongly hyperbolic. In fact, they proved with some additional work that it suffices to prove such a family is hyperbolic. The proof of that fact however was intentionally omitted from their paper as it relies on the proof of Theorem~\ref{Cylinder}; we provide the proof of said hyperbolicity in Section~\ref{Exp}, thereby completing the proof of Theorem~\ref{ExpSurface2}.

\subsection{Outline of Proof}\label{Outline}

The above theorems show the value in establishing that the family of $4$-list-critical graphs is strongly hyperbolic. Our main result - Theorem~\ref{Cylinder} - proves this fact. To prove Theorem~\ref{Cylinder}, we will need Thomassen's~\cite{ThomShort} stronger inductive statement as follows.

\begin{theorem}\label{Thom0}
Let $G$ be a plane graph of girth at least $5$ and $C$ the outer cycle of $G$.
Let $P$ be a path in $G$ of length at most $5$, such that $V(P)\subseteq V(C)$. Let
$L$ be an assignment of lists to the vertices of $G$ such that $|L(v)|$ = 3 for all
$v \in V(G)\setminus V(C)$, $|L(v)|\ge 2$ for all $v\in V(C) \setminus V(P)$, and $|L(v)| = 1$ for all $v\in V (P)$. Further suppose that no vertex $v$ with $|L(v)| = 2$ is adjacent to a vertex $u$ such that $|L(u)|\le 2$. If
there exists an $L$-coloring of the subgraph induced by $V(P)$, then there exists an $L$-coloring of $G$.
\end{theorem}

To prove Theorem~\ref{Cylinder}, we then consider the structures which arise from Theorem~\ref{Thom0}. This is also the idea behind Dvo\v{r}\'ak and Kawarabayashi's proof. However, they used a stronger version of Theorem~\ref{Thom0} (which they also proved) to yield a shorter list of structures and from these derived an inductive formula on the size of a $C$-critical graph which decreases if there are long faces (an idea also used by Dvo\v{r}\'ak, Kr\'al' and Thomas in~\cite{DvoKraTho2} and~\cite{DvoKraTho3}). 

For two cycles, we are not able to use Dvo\v{r}\'ak and Kawarabayashi's stronger version of Theorem~\ref{Thom0}. Instead, we rely only on one additional result of Thomassen~\cite{ThomRegular}, which is the key to his proof that there are finitely many $4$-critical graphs of girth at least five embeddable in a fixed surface. That result (stated in this paper as Theorem~\ref{CylinderStructure}) says that two cycles that are far apart (distance at least three) and whose vertices of lists of size two form an independent set has a $3$-list-coloring (technically Thomassen's proof is done in terms of ordinary coloring but it is easily adapted to work for list-coloring).

Thus our proof for two cycles must only use the structures arising from Theorem~\ref{Thom0}. Hence we also provide an independent proof of Dvo\v{r}\'ak and Karawabayashi's result. To accomplish this, we also prove a general inductive formula on the size of a $C$-critical graph which decreases if there are long faces; crucially though, the formula also decreases if there are many edges from vertices in $C$ to vertices not in $C$. That subtlety is enough to allow us to use the weaker list of structures arising from Theorem~\ref{Thom0}.

\subsection{Outline of Paper}
In Section~\ref{Canvas}, we provide some necessary preliminaries and list the structures arising in $C$-critical graphs (see Lemma~\ref{Structure3}). In Section~\ref{Linear}, we develop our general inductive formula (see Definition~\ref{Parameters} and Theorem~\ref{StrongLinear}) on the size of a $C$-critical graph and then show how it implies Theorem~\ref{Cylinder}. In Section~\ref{Proof}, we prove said general formula. Finally, in Section~\ref{Exp}, we provide as promised the proof that the family of exponentially critical graphs is hyperbolic thereby completing the proof of Theorem~\ref{ExpSurface2} (see Theorem~\ref{ExpHyper}).

\section{Critical Canvases}\label{Canvas}

In this section, we develop the necessary preliminaries and provide a key structural lemma (Lemma~\ref{Structure3}). Let us first define the graphs we will be working with as follows.

\begin{definition}[Canvas]
We say that $(G, S, L)$ is a \emph{canvas} if $G$ is a plane graph of girth at least five, $S$ is a subgraph of $G$, $L$ is a list assignment for the vertices of $G$ such that $|L(v)|\ge 3$ for all $v\in V(G)\setminus V(S)$ and there exists an $L$-coloring of $S$. We call $S$ the \emph{boundary} of the canvas. We say a canvas $(G,S,L)$ is \emph{critical} if $G$ is $S$-critical with respect to the list assignment $L$.
\end{definition}

We need the following lemma about subgraphs of critical graphs. The lemma is standard and can be found in~\cite{PostleThomas2} but we include its proof for completeness. Note that if $G$ is a graph and $A,B$ are subgraphs of $G$, we let $A\cap B$ denote the graph where $V(A\cap B)=V(A)\cap V(B)$ and $E(A\cap B)=E(A)\cap E(B)$.

\begin{lemma} \label{SComponent}
Let $S$ be a subgraph of a graph $G$ such that $G$ is $S$-critical graph with respect to a list assignment $L$. Let $A,B\subseteq G$ such that $A\cup B=G, S\subseteq A$ and $B\ne A\cap B$. Then $B$ is $A\cap B$-critical.
\end{lemma}
\begin{proof}
Since $G$ is $S$-critical, every isolated vertex of $G$ is in $S$, and
thus every isolated vertex of $B$ is in $A\cap B$. Suppose for a contradiction
that $B$ is not $A\cap B$-critical. Then, there exists an edge $e \in E(B) \setminus E(A\cap B)$ such
that every $L$-coloring of $A\cap B$ that extends to $B \setminus e$ also extends to $B$. 

Note that $e \not\in E(S)$. Since $G$ is $S$-critical, then there exists an $L$-coloring $\phi$ of $S$ that
extends to an $L$-coloring $\phi$ of $G \setminus e$, but does not extend to an $L$-coloring
of $G$. However, by the choice of $e$, the restriction of $\phi$ to $A\cap B$ extends to an
$L$-coloring $\phi'$ of $B$. Let $\phi''$ be the coloring that matches $\phi'$ on $V(B)$ and $\phi$
on $V(G) \setminus V(B)$. Observe that $\phi''$ is an $L$-coloring of $G$ extending $\phi$, which
is a contradiction.
\end{proof}

Lemma~\ref{SComponent} has two useful corollaries. To state it, we need the following definitions.

\begin{definition}
Let $T=(G,S,L)$ be a canvas and $H\subseteq G$ such that $S\subseteq H$. We define the \emph{subcanvas} of $T$ induced by $H$ to be $(H,S,L)$, which we denote $T|H$. Similarly, we define the \emph{supercanvas} of $T$ induced by $H$ to be $(G,H,L)$, which we denote $T/H$.
\end{definition}

First let us note a useful fact about canvases whose proof we omit.

\begin{prop}\label{CriticalSubgraph}
Let $T=(G,S,L)$ be a canvas. If there exists a proper $L$-coloring of $S$ that does not extend to $G$, then $T$ contains a critical subcanvas.
\end{prop}

Here then is the first corollary of Lemma~\ref{SComponent}.

\begin{cor} \label{Subcanvas}
Let $T=(G,S,L)$ be a critical canvas. If $H$ is a subgraph of $G$ such that $S$ is a proper subgraph of $H$ and every edge in $E(G)\setminus E(H)$ has no end in $V(H)\setminus V(S)$, then the subcanvas induced by $H$, that is $T|H$, is a critical canvas.
\end{cor}
\begin{proof}
This follows from Lemma~\ref{SComponent} with $B=H$ and $A=G\setminus (V(H)\setminus V(S))$ since $A\cap B=S$ and $A\cup B=G$.
\end{proof}

Finally here is the second corollary of Lemma~\ref{SComponent}.

\begin{cor}\label{Supercanvas}
Let $T=(G,S,L)$ be a critical canvas. If $H$ is a proper subgraph of $G$ containing $S$, then the supercanvas induced by $H$, that is $T/H$, is critical.
\end{cor}
\begin{proof}
This follows from Lemma~\ref{SComponent} with $B=G$ and $A=H$.
\end{proof}

\subsection{Critical Canvases with One Boundary Component}

We now prove a structure theorem for critical canvases. For that, we need the following structures.

\begin{definition}
Let $T=(G,S,L)$ be a canvas. Let $k\ge 1$ and let $P=p_1p_2\ldots p_{k+1}$ be a path in $G\setminus V(S)$. We say $P$ is a

\begin{itemize}
\item a \emph{neighboring $k$-path} of $S$ if $p_i\in N(S)$ for all $i$ with $1\le i \le k+1$,
\item a \emph{semi-neighboring $3$-path} of $S$ if $k=3$ and $p_1,p_2,p_4\in N(S)$,
\item a \emph{semi-neighboring $5$-path} of $S$ if $k=5$ and $p_1,p_2,p_5,p_6\in N(S)$.
\end{itemize}

Finally let $v\in N(S)\setminus S$ and suppose that $v$ has three neighbors $u_1,u_2,u_3$ in $N(S)\setminus S$. Then we say that $G[\{v,u_1,u_2,u_3\}]$ is a \emph{neighboring claw} of $S$.

\end{definition}

We can first derive a set of simple structures from Theorem~\ref{Thom0} for critical canvases whose boundary has one component as follows.

\begin{lemma}\label{Structure}
If $T=(G,S,L)$ is a critical canvas such that $S$ is connected, then there exists one of the following:

\begin{enumerate}
\item an edge not in $S$ with both ends in $V(S)$, or,
\item a vertex not in $V(S)$ with at least two neighbors in $V(S)$, or,
\item a neighboring $1$-path of $S$.
\end{enumerate}
\end{lemma}
\begin{proof}
Suppose not. Let $\phi$ be an $L$-coloring of $S$ that does not extend to $G$. Let $L'(v)=L(v)\setminus \{\phi(u)| u\in N(v)\cap V(S)\}$. Since 1 does not hold, $V(G)\setminus V(S)\ne\emptyset$. Since 2 does not hold, $|L'(v)|\ge 2$ for all $v\in V(G)\setminus V(S)$. Since 3 does not hold, there does not exist an edge $uv \in E(G)$ with $u,v \in V(G)\setminus V(S)$ with $|L'(u)|=|L'(v)|=2$. But then by Theorem~\ref{Thom0} applied to $G\setminus V(S)$ and $L'$ (with $P=\emptyset$), there exists an $L'$-coloring $\phi'$ of $G$. But now $\phi''=\phi\cup\phi'$ is an $L$-coloring of $G$, a contradiction.
\end{proof}

Lemma~\ref{Structure} is not useful for our proof however, since a neighboring $1$-path is unhelpful for reductions. Nevertheless, by coloring neighboring $1$-paths and using a second application of Theorem~\ref{Thom0} (or rather Theorem~\ref{Structure}), we can deduce a stronger outcome than Lemma~\ref{Structure}(3) as follows.

\begin{lemma}\label{StructureB}
If $T=(G,S,L)$ is a critical canvas such that $S$ is connected, then there exists one of the following:

\begin{enumerate}
\item an edge not in $S$ with both ends in $V(S)$, or,
\item a vertex not in $V(S)$ with at least two neighbors in $V(S)$, or,
\item a neighboring $2$-path of $S$, or
\item a semi-neighboring $3$-path of $S$, or,
\item a semi-neighboring $5$-path of $S$.
\end{enumerate}
\end{lemma}
\begin{proof}
Suppose not. Since 3 does not hold, the components of $N(S)\setminus S$ have size at most two.  Let $R$ be the union of all components of $N(S)\setminus S$ of size at most two. 
As 1 and 2 do not hold, it follows from Lemma~\ref{Structure} applied to $T$ that there exists a neighboring $1$-path of $S$ and hence $R\ne \emptyset$. Let $H=G[V(S)\cup V(R)]$. As the vertices in $R$ have degree at most two in $H$, it follows that $H$ is a proper subgraph of $G$. By Corollary~\ref{Supercanvas}, $T/H$ is a critical canvas. Note that $H$ is connected and hence we can apply Lemma~\ref{Structure} to $T/H$.

Note that Lemma~\ref{Structure}(1) does not hold for $T/H$ as $H$ is an induced subgraph of $G$. So let us suppose that Lemma~\ref{Structure}(2) holds for $T/H$, that is there a vertex $v$ not in $V(H)$ with at least two neighbors in $V(H)$. Let $u,w$ be neighbors of $v$ in $V(H)$. As 2 does not hold for $T$, at least one of $u$ or $w$ must be in $R$. Suppose without loss of generality that $u$ is in $R$. Let $u'$ be the unique neighbor of $u$ in $R$. As $G$ has girth at least five, $u'\ne w$. As 3 does not hold for $T$, it follows that $w\in R$. But now $u'uvw$ is a semi-neighboring $3$-path of $S$ in $T$, that is 4 holds, a contradiction.

So we may assume that Lemma~\ref{Structure}(3) holds for $T/H$, that is there is a neighboring $1$-path $P=p_1p_2$ of $H$ in $T/H$. Let $u_1$ be a neighbor of $p_1$ in $V(H)$ and let $u_2$ be a neighbor of $p_2$ in $V(H)$. Since $p_1,p_2\notin V(R)$, at least one of $u_1$ or $u_2$ is not in $S$. Suppose without loss of generality that $u_1$ is not in $S$. Hence $u_1\in R$. Let $u_1'$ be the unique neighbor of $u_1$ in $R$. If $u_2\in V(S)$, then $u_1'u_1p_1p_2$ is a semi-neighboring $3$-path of $S$ in $T$ and 4 holds, a contradiction. So we may assume $u_2\in R$. Let $u_2'$ be the unique neighbor of $u_2$ in $R$. Note that as $G$ has girth at least five, $u_1\ne u_2$ and $u_1$ is not adjacent to $u_2$. Hence $u_1,u_2,u_1',u_2'$ are all distinct. But then $u_1'u_1p_1p_2u_2u_2'$ is a semi-neighboring $5$-path of $S$ in $T$ and 5 holds, a contradiction.
\end{proof}

Unfortunately, in our proof neighboring $2$-paths are also not strong enough for reductions. Yet neighboring $2$-paths just barely fail in this regard. To that end, we make the following definition.

\begin{definition}
Let $T=(G,S,L)$ be a canvas. Let $P=p_1p_2p_3$ be a neighboring $2$-path of $S$ such that for each $i\in\{1,2,3\}$, $p_i$ has a unique neighbor $u_i$ in $S$. Let $H=G\cup P \cup \{p_1u_1,p_2u_2,p_3u_3\}$. We say $T/H$ is obtained from $T$ by \emph{relaxing} $P$ and that $T/H$ is a \emph{relaxation} of $T$. We define $T$ to be a $0$-relaxation of itself. For $k\ge 1$, we say a supercanvas $T'$ of $T$ is a \emph{$k$-relaxation} of $T$ if $T'$ is a relaxation of a $(k-1)$-relaxation of $T$.
\end{definition}

Now by first coloring neighboring $2$-paths and applying Lemma~\ref{StructureB} a second time, we can upgrade the outcomes of Lemma~\ref{StructureB} (in particular outcome 3) at the cost of finding outcomes 4 or 5 in a $k$-relaxation for some $k\le 2$ as follows.

\begin{lemma}\label{Structure2}
If $T=(G,S,L)$ is a critical canvas such that $S$ is connected, then there exists one of the following:

\begin{enumerate}
\item an edge not in $S$ with both ends in $S$, or,
\item a vertex not in $V(S)$ with at least two neighbors in $V(S)$, or,
\item a neighboring claw of $S$, or
\item a $k$-relaxation $T'=(G,S',L)$ of $T$ with $k\le 2$ and a semi-neighboring $3$-path of $S'$, or,
\item a $k$-relaxation $T'=(G,S',L)$ of $T$ with $k\le 2$ and a semi-neighboring $5$-path of $S'$.
\end{enumerate}
\end{lemma}
\begin{proof}
Suppose not. Note that since $G$ has no semi-neighboring $3$-path of $S$ then $G$ has no neighboring $k$-path of $S$ for any $k\ge 3$. Let $R$ be the union of all components of $N(S)\setminus S$ of size exactly three. By Lemma~\ref{StructureB}, we may assume that $R\ne\emptyset$ as otherwise one of 1, 2, 4 or 5 holds, a contradiction. 

Let $H=G[V(C)\cup V(R)]$. As there is no neighboring claw or neighboring $3$-path of $S$, every vertex in $V(H)\setminus V(S)$ is in a unique neighboring $2$-path of $S$; let $P(u)$ denote said path for each $u\in V(H)\setminus V(S)$. Further note that if $Q$ is a neighboring $1$-path of $H$, then either $Q$ is a neighboring $1$-path of $S$ or the neighbors of $Q$ in $H$ are contained in a unique neighboring $2$-path of $C$, as otherwise there exists a semi-neighboring $3$-path or semi-neighboring $5$-path of $C$, that is 4 or 5 holds, a contradiction. In the latter case, let $P(Q)$ denote this unique neighboring $2$-path of $S$.

Note that $H$ is a proper subgraph of $G$ as there exist vertices in $V(H)\setminus V(S)$ which degree two in $H$ but degree at least three in $G$. By Lemma~\ref{Supercanvas}, $T/H$ is critical. 
Apply Lemma~\ref{StructureB} to $T/H$. Note that Lemma~\ref{StructureB}(1) does not hold for $T/H$ as $H$ is an induced subgraph of $G$ since 1 and 2 do not hold for $T$ and $R$ is a set of components of $N(S)\setminus S$. 

So suppose Lemma~\ref{StructureB}(1) holds, that is there exists a vertex $v\in V(G)\setminus V(H)$ with two neighbors $u_1,u_2$ in $V(H)$. Since 2 does not hold for $T$ and $R$ is a set of components of $N(S)\setminus S$, it follows that $u_1,u_2\not\in V(S)$ and hence $u_1,u_2\in V(R)$. As $G$ has girth at least five, $P(u_1)\ne P(u_2)$. But then there exists a semi-neighboring $3$-path of $S$ contained in $P(u_1)\cup P(u_2) \cup \{v\}$ and hence 4 holds, a contradiction.

Next suppose Lemma~\ref{StructureB}(3) holds, that is there exists a neighboring $2$-path $P=p_1p_2p_3$ of $H$. Given that $V(P)\cap V(R)=\emptyset$, there exists $i\in \{1,2,3\}$ such that $p_i\in N(R)$. We may suppose without loss of generality that $i\in\{1,2\}$. But then the neighbors of $p_1p_2$ are contained in a unique neighboring $2$-path of $C$, $P(p_1p_2)$ as noted above. As $G$ has girth at least five, the neighbor of $p_1$ and the neighbor of $p_2$ must be the ends of $P(p_1p_2)$. Yet now the neighbors of $p_2p_3$ are contained in a unique neighboring $2$-path of $C$, $P(p_2p_3)$ and we find that $P(p_1p_2)=P(p_2p_3)$. As $G$ has girth at least five, the neighbor of $p_3$ in $P(p_1p_2)$, call it $x$, is distinct from the neighbor of $p_2$ in $P(p_1p_2)$. But then $xp_1p_2p_3$ is a $4$-cycle, contradicting that $G$ has girth at least five.

Next suppose Lemma~\ref{StructureB}(4) holds, that is there exists a semi-neighboring $3$-path $P=p_1p_2p_3p_4$ of $H$, where $p_1,p_2,p_4\in N(V(H))\setminus V(H)$. As $p_1p_2$ is a neighboring $1$-path of $H$, the neighbors of $p_1p_2$ are either contained in a unique neighboring $2$-path of $S$, $P(p_1p_2)$, or $p_1p_2$ is a neighboring $1$-path of $S$. Let $y$ be the neighbor of $p_4$ in $V(H)$. Now either $y\in V(S)$ or $y\in V(R)$. In all cases, $P$ is a semi-neighboring $3$-path of a $\le 2$-relaxation $T'$ of $T$, where either $T'=T$, or $T'$ is obtained from $T$ by relaxing $P(p_1p_2)$, or by relaxing $P(y)$ or by relaxing both. Hence 4 holds, a contradiction.

Next suppose Lemma~\ref{StructureB}(5) holds, that is there exists a semi-neighboring $5$-path $P=p_1p_2p_3p_4p_5p_6$ of $H$, where $p_1,p_2,p_5,p_6\in N(V(H))\setminus V(H)$. As $p_1p_2$ is a neighboring $1$-path of $H$, the neighbors of $p_1p_2$ are either contained in a unique neighboring $2$-path of $S$, $P(p_1p_2)$, or $p_1p_2$ is a neighboring $1$-path of $S$. Similarly, as $p_5p_6$ is a neighboring $1$-path of $H$, the neighbors of $p_5p_6$ are either contained in a unique neighboring $2$-path of $S$, $P(p_5p_6)$, or $p_5p_6$ is a neighboring $1$-path of $S$. In all cases, $P$ is a semi-neighboring $5$-path of a $\le 2$-relaxation $T'$ of $T$, where either $T'=T$, or $T'$ is obtained from $T$ by relaxing $P(p_1p_2)$, or by relaxing $P(p_5p_6)$ or by relaxing both. Hence 5 holds, a contradiction.

\end{proof}

\subsection{Critical Canvases with Two Boundary Components}

We will need a similar structural lemma for critical canvases whose boundary has two components. This can be done using the following theorem of Thomassen~\cite{ThomRegular}.

\begin{thm}[Theorem 5.1 in~\cite{ThomRegular}]\label{CylinderStructure}
Let $G$ be a connected plane graph of girth at least five, $C_1,C_2$ the boundaries of distinct faces of $G$, and $L$ a $3$-list-assignment for $G$ such that $|L(v)|\ge 2$ for all $v\in V(G)$ and $|L(v)|\ge 3$ for all $v\in V(G)\setminus (V(C_1)\cup V(C_2))$. Then there is an $L$-coloring of $G$ unless one of the following holds:
\begin{enumerate}
\item[(iii)] $G$ has a path $u_1u_2u_3$ each vertex of which has precisely two available colors, or
\item[(iv)] $G$ has a path $u_1u_2u_3u_4$ such that each of $u_1, u_2, u_4$ has precisely two available
colors, or
\item[(v)] $G$ has a path $u_1u_2u_3u_4u_5u_6$ such that each of $u_1, u_2, u_5, u_6$ has precisely two
available colors, or
\item[(vi)] $G$ has a path $w_1w_2$ or $w_1xw_2$ such that $w_i$ is in $C_i$ and $|L(w_i)|=2$ for $i=1,2$, or
\item[(vii)] $G$ has a path $w_1w_2w_3w_4w_5$ such that $w_1, w_2$ are in one of $C_1,C_2$, $w_5$ is in the other,
and $|L(w_1)|=|L(w_2)|=|L(w_5)|=2$, or
\item[(viii)] $G$ has a path $w_1w_2w_3w_4w_5w_6w_7$ such that $w_1, w_2$ are in one of $C_1,C_2$, $w_6,w_7$ are in the other,
and $|L(w_1)|=|L(w_2)|=|L(w_6)|=|L(w_7)|=2$.
\end{enumerate}
\end{thm}

The proof of Theorem~\ref{CylinderStructure} is done for ordinary coloring but it can be easily modified to give the same result for list-coloring. More precisely, the proof at times uses the fact that list of vertices are the same so as to identify vertices or color them with the same color; when the lists are not the same, this is not always possible, but in such cases the proof can be modified to avoid these assumptions. Further note that condition $(viii)$ was erroneously omitted from the statement of Theorem 5.1 in~\cite{ThomRegular} (and is necessary for the first inductive argument). 

Using Theorem~\ref{CylinderStructure}, we can now generalize Lemma~\ref{StructureB} to critical canvases whose boundary has two components as long as the distance between those components is at least $7$.

\begin{lemma}\label{StructureC}
If $T=(G,S,L)$ is a critical canvas such that $S$ has at most two components and if two then they are distance at least $7$, then there exists one of the following:

\begin{enumerate}
\item an edge not in $S$ with both ends in $V(S)$, or,
\item a vertex not in $V(S)$ with at least two neighbors in $V(S)$, or,
\item a neighboring $2$-path of $S$, or
\item a semi-neighboring $3$-path of $S$, or,
\item a semi-neighboring $5$-path of $S$.
\end{enumerate}
\end{lemma}
\begin{proof}
Suppose not and let $T$ be a counterexample such that $|V(G)|$ is minimized. By Lemma~\ref{StructureB}, we may assume that $S$ has at least two components $S_1$ and $S_2$. 

First suppose that $G$ is not connected. As $G$ is critical, by Theorem~\ref{Thom0}, $G$ has two components $G_1,G_2$ such that $S_i\subseteq G_i$ for $i\in\{1,2\}$. Since $T$ is critical, for at least one $i\in\{1,2\}$, $S_i\ne G_i$. Let us assume without loss of generality that $S_1\ne G_1$. By Lemma~\ref{SComponent} applied to $G$ with $B=G_1$ and $A=S_1$, we find that $G_1$ is $S_1$-critical. Let $T_1=(G_1,S_1,L)$. Now $T_1$ is a critical canvas. As $S_1$ is connected, Lemma~\ref{StructureB} applied to $T_1$ implies that one of 1-5 holds, a contradiction. So we may assume that $G$ is connected.

Next suppose that $G\setminus V(S)$ is not connected. Let $H$ be a component of $G\setminus V(S)$. By Lemma~\ref{Subcanvas}, $T'=T|(G[V(H)\cup V(S)])$ is a critical canvas and yet $|V(G[V(H)\cup V(S)])|$ is smaller than $|V(G)|$. Hence one of 1-5 holds for $T'$ and hence for $T$, a contradiction. So we may assume that $G\setminus V(S)$ is connected.

As 1 does not hold, $S$ is an induced subgraph of $G$. Let $\phi$ be an $L$-coloring of $S$ that does not extend to $G$. Let $L'(v)=L(v)\setminus \{\phi(u)| u\in N(v)\cap V(C)\}$. As $\phi$ does not extend to $G$, there does not exist an $L'$-coloring of $G\setminus V(S)$. As 2 does not hold, $|L'(v)|\ge 2$ for all $v\in V(G')$. Let $C_1,C_2$ be the boundaries of the two faces in $G'$ which contain vertices of $S$ in their interior. Note that $|L'(v)|\ge 3$ for all $v\in V(G')\setminus (V(C_1)\cup V(C_2))$. 

Note that as 1 and 2 not hold, $N(V(S_1))\cap N(V(S_2)) = \emptyset$. Moreover $N(V(S_1))\setminus V(S_1)\ne \emptyset$ as $G$ is connected. Similarly $N(V(S_2))\setminus V(S_2) \ne \emptyset$. Let $G'$ be obtained from $G\setminus V(S)$ by adding a path of four new degree two vertices inside $C_1$ between consecutive (in the cyclic order) pairs of vertices in $N(V(S_1))\setminus V(S_1)$ and similarly in $C_2$ for pairs of vertices in $N(V(S_2))\setminus V(S_2)$ (if there is only one such vertex, then we add a path to itself). For each vertex $v$ in $V(G')\setminus (V(G)\setminus V(S))$, let $L'(v)$ be a list of three arbitrary colors.

Note that $G'$ is a connected plane graph of girth at least five and that $|L'(v)|\ge 2$ for all $v\in V(G')$. Moreover, there exists two faces $C_1',C_2'$ of $G'$ such that every vertex in $G'$ with $|L(v')|\ge 2$ is on the boundary of $C_1'$ or $C_2'$. Moreover, no vertex in $G'$ is on the boundary of both $C_1'$ and $C_2'$ by construction. Finally note that there does not exist an $L'$-coloring of $G'$ as otherwise there exists an $L'$-coloring of $G\setminus V(S)$, a contradiction. 

By Theorem~\ref{CylinderStructure}, as there is no $L'$-coloring of $G'$, one of Theorem~\ref{CylinderStructure}(iii)-(viii) holds. As the distance between $S_1$ and $S_2$ is at least 7, none of (vi)-(viii) holds as otherwise there exists a path from a neighbor of a vertex in $S_1$ to a neighbor of a vertex in $S_2$ of length at most 4 and hence a path from $S_1$ to $S_2$ of length 6. If (iii) holds, then there exists a neighboring $2$-path of $S$ and 3 holds, a contradiction. If (iv) holds, then there exists a semi-neighboring $3$-path of $S$ and 4 holds, a contradiction. If (v) holds, then there exists a semi-neighboring $5$-path of $S$, a contradiction.
\end{proof}

Finally we generalize Lemma~\ref{Structure2} to canvases whose boundary has two components as long as the distance between those components is at least $9$.

\begin{lemma}\label{Structure3}
If $T=(G,S,L)$ is a critical canvas such that $S$ has at most two components and if two then the distance between the components is at least $9$, then there exists a $k$-relaxation $T'=(G,S',L)$ of $T$ with $k\le 2$ such that there exists one of the following:

\begin{enumerate}
\item an edge not in $S$ with both ends in $S$, or,
\item a vertex not in $S$ with at least two neighbors in $S$, or,
\item a neighboring claw of $S$, or
\item a semi-neighboring $3$-path of $S'$, or,
\item a semi-neighboring $5$-path of $S'$.
\end{enumerate}
\end{lemma}
\begin{proof}
Suppose not. Note that since $G$ has no semi-neighboring $3$-path of $S$ then $G$ has no neighboring $k$-path of $S$ for any $k\ge 3$. Let $R$ be the union of all components of $N(S)\setminus S$ of size exactly three. By Lemma~\ref{StructureC}, we may assume that $R\ne\emptyset$ as otherwise one of 1, 2, 4 or 5 holds, a contradiction. 

Let $H=G[V(C)\cup V(R)]$. As there is no neighboring claw or neighboring $3$-path of $S$, every vertex in $V(H)\setminus V(S)$ is in a unique neighboring $2$-path of $S$; let $P(u)$ denote said path for each $u\in V(H)\setminus V(S)$. Note that $H$ is a proper subgraph of $G$ as there exist vertices in $V(H)\setminus V(S)$ which degree two in $H$ but degree at least three in $G$. By Lemma~\ref{Supercanvas}, $T/H$ is critical. The proof now proceeds identically as Lemma~\ref{Structure2} except that we apply Lemma~\ref{StructureC} instead of Lemma~\ref{StructureB}, which is permissible since if $H$ has two components then they are at distance at least 7 because if $S$ had two components then they were at distance at least 9 by assumption.
\end{proof}

\section{Linear Bound for Two Cycles}\label{Linear}

In this section, we develop the parameters necessary to state our general formula (Theorem~\ref{StrongLinear}), state said formula and derive Theorem~\ref{Cylinder} from it. As for the proof of Theorem~\ref{StrongLinear}, it comprises the entirety of Section~\ref{Proof}. First a few definitions.

\subsection{Deficiency}

First we need the following key parameter which essentially tracks how many edges the graph of a canvas is below the maximum imposed by Euler's formula on a planar graph of girth five; hence the parameter is larger if the graph has many faces of length more than five.

\begin{definition}
Let $T=(G,S,L)$ be a canvas. We let $v(T):= |V(G)\setminus V(S)|$ and $e(T) := |E(G)\setminus E(S)|$. We define the \emph{deficiency} of the canvas $T$ as

$${\rm def}(T):=3e(T) - 5v(T) + 10(c(S)-c(G)),$$

\noindent where $c(S)$ is the number of components of $S$ and $c(G)$ is the number of components of $G$.
\end{definition}

We now prove the following very useful lemma, which says that the deficiency of a canvas equals the sum of the deficiencies of a subcanvas and its supercanvas.

\begin{lem} \label{defsum}
If $T=(G,C,L)$ is a canvas and $H$ is a subgraph of $G$ containing $S$, then

$${\rm def}(T) = {\rm def}(T|H) + {\rm def}(T/H).$$
\end{lem}  

\begin{proof}
Every edge of $E(G)\setminus E(S)$ is in exactly one of $E(H)\setminus E(S)$ or $E(G)\setminus E(H)$. Similarly, every vertex of $V(G)\setminus V(S)$ is in exactly one of $V(H)\setminus V(S)$ or $V(G)\setminus V(H)$. Lastly $c(S)-c(G) = c(S)-c(H) + c(H)-c(G)$. Combining these facts gives the desired formula.
\end{proof}

\subsection{More Complicated Parameters}

Next we will need more complicated parameters. As alluded to in Section~\ref{Outline}, we need to track the number of edges not in $S$ with an end in $S$. Also to show that the number of vertices is bounded, we will need to add a small additional weight to the vertices above and beyond what is already counted in deficiency. Thus we will need two small weights, $\alpha$ for tracking the edges out of $S$, and $\epsilon$ for the vertices. 

We will prove our general formula for critical canvases (Theorem~\ref{StrongLinear}) assuming a number of inequalities on $\alpha,\epsilon$. Then to prove Theorem~\ref{Cylinder}, we deduce the appropriate $\alpha$ and $\epsilon$ for the formula to hold. So for the benefit of the reader, we shall assume these are fixed but unspecified constants except in deriving Theorem~\ref{Cylinder} when it is needed to specify them. Strangely, while the formula holds for any small enough $\epsilon$, the value of $\alpha$ is more tightly controlled and needs to be slightly between $1/3$and $2/5$ (in fact any value in $(1/3,2/5)$ is acceptable if $\epsilon$ is made small enough).  

\begin{definition}\label{Parameters}
Let $T=(G,S,L)$ be a canvas. Fix $\epsilon,\alpha > 0$. We define 

$$q(T):= \sum_{v\in V(S)} {\rm deg}_{G-E(S)}(v),$$
$$s(T):=\epsilon v(T) + \alpha q(T),$$ 
$$d(T):={\rm def}(T)-s(T).$$
\end{definition}

Thus $q(T)$ equals the number of edges not in $S$ with an end in $S$ (where an edge with both ends in $S$ is counted twice). We now prove that these new parameters satisfy natural relations for subcanvases and supercanvases as in Lemma~\ref{defsum}.

\begin{prop}\label{surplussum}
Let $T=(G,S,L)$ be a canvas and $H$ a subgraph of $G$ containing $S$. The following statements hold:

\begin{itemize}
\item $v(T)=v(T|H)+v(T/H)$,
\item $q(T)\le q(T|H)+q(T/H)$,
\item $s(T)\le s(T|H)+s(T/H)$,
\item $d(T)\ge d(T|H)+d(T/H)$.
\end{itemize}
\end{prop}
\begin{proof}
The first statement follows as every vertex of $V(G)\setminus V(S)$ is in exactly one of $V(H)\setminus V(S)$ or $V(G)\setminus V(H)$. To prove the second statement, note that 

$$q(T) = \sum_{v\in V(S)} {\rm deg}_{G-E(S)}(v) = \sum_{v\in V(S)} {\rm deg}_{G-E(H)}(v) + \sum_{v\in V(S)} {\rm deg}_{H-E(S)}(v).$$

\noindent But then as $V(S)\subseteq V(H)$,

$$q(T) \le  \sum_{v\in V(H)} {\rm deg}_{G-E(H)}(v) + \sum_{v\in V(S)} {\rm deg}_{H-E(S)}(v) = q(T/H) + q(T|H),$$ 

\noindent which proves the second statement. The third statement follows from the first two. The fourth statement follows from the third and Lemma~\ref{defsum}. 
\end{proof} 

We can improve upon Proposition~\ref{surplussum} by noting an improved bound on $q(T)$. First, a definition.

\begin{definition}
If $T=(G,S,L)$ is a canvas and $H$ is a subgraph of $G$ containing $S$, then we let 

$$q_T(H,S):=\sum_{v\in V(H)\setminus V(S)}{\rm deg}_{G-E(H)}(v)$$

and

$$d_T(T|H) := d(T|H) + \alpha q_T(H,S).$$
\end{definition}

Note that if $T$ is critical, then the vertices in $V(H)\setminus V(S)$ have degree at least three. Hence, $q(H,S)$ is at least the number of vertices in $V(H)\setminus V(S)$ of degree two in $H$ plus twice the number of vertices of degree one in $H$.

\begin{prop}\label{surplussum2}
If $T=(G,S,L)$ is a canvas and $H$ is a subgraph of $G$ containing $S$, then

$$q(T) = q(T/H) + q(T|H) - q_T(H,S),$$

and hence

$$d(T)  \ge d_T(T|H) + d(T/H).$$
\end{prop}
\begin{proof}
As in the proof of Proposition~\ref{surplussum}, we have that

$$q(T) = \sum_{v\in V(S)} {\rm deg}_{G-E(H)}(v) + \sum_{v\in V(S)} {\rm deg}_{H-E(S)}(v).$$

But then 

$$q(T) =  q(T/H)+q(T|H) - \sum_{v\in V(H)\setminus V(S)}{\rm deg}_{G-E(H)}(v) = q(T/H) + q(T|H) - q_T(H,S).$$ 

\end{proof}

\subsection{Small Canvases and a General Formula}

The next proposition determines $d$ for small canvases.
\begin{prop} \label{d0}
Let $T=(G,S,L)$ be a canvas.
\begin{enumerate}
\item[(i)] If $G=S$, then $d(T)= 0$. 
\item[(ii)] If $v(T)=0$, then $d(T)\ge (3-2\alpha)e(T)$.
\item[(iii)] If $v(T)=1$, $c(G)=c(S)$ and $v\in V(G\setminus S)$, then $d(T)\ge (3-2\alpha)e(T)-5+\deg(v)\alpha-\epsilon$.
\end{enumerate}
\end{prop}
\begin{proof}
If $G=S$, then $v(T)=q(T)=s(T)=0$. Moreover, $c(G)=c(S)$ and hence ${\rm def}(T) = 0$. So $d(T)= 0-0=0$ as desired. This proves (i).

If $v(T)=0$, then $q(T)=|E(G)\setminus E(S)|$. Thus $s(T)=2\alpha|E(G)\setminus E(S)|=2\alpha e(T)$. As $v(T)=0$, $c(G)\le c(S)$ and hence ${\rm def}(T)\ge 3e(T)$. So $d(T)\ge (3-2\alpha)e(T)$ as desired. This proves (ii).

Let $v\in V(G)\setminus V(S)$. As $v(T)=1$, $q(T)=2|E(G)\setminus E(S)|-\deg(v)=2e(T)-\deg(v)$. Thus $s(T)=\alpha(2e(T)-\deg(v))+\epsilon$.  Yet as $c(G)=c(S)$, ${\rm def}(T)\ge 3e(T)-5v(T) = 3e(T)-5$. Combining, $d(T)\ge(3-2\alpha)e(T)-5+\deg(v)\alpha-\epsilon$ as desired. This proves (iii).
\end{proof}

In particular, Proposition~\ref{d0} says that if $G$ consists of $S$ and one edge that is not in $S$ and yet has both ends in $S$, then $d(T)\ge 3-2\alpha$. Similarly if $G$ consists of $S$ and one vertex not in $S$ of degree three, then $d(T)\ge 4-3\alpha-\epsilon$. These two critical canvases are special for our proof in that they have the smallest value of $d$. To that end, we make the following definitions.

\begin{definition}
Let $T=(G,S,L)$ be a canvas. We say $T$ is a \emph{chord} if $G$ consists of exactly $S$ and one edge not in $S$ with both ends in $S$. We say $T$ is a \emph{tripod} if $G$ consists of exactly $S$ and one vertex not in $S$ of degree three. We say $T$ is \emph{singular} if $T$ is a chord or a tripod and \emph{non-singular} otherwise. We say $T$ is \emph{normal} if no subcanvas of $T$ is singular.
\end{definition}

We are now ready to state our generalization of the linear bound for two cycles. It asserts that the only exceptions are the two cases listed above.

\begin{thm}\label{StrongLinear}
Let $\epsilon, \alpha > 0$ satisfy the following where $\epsilon,\alpha$ are as in the definition of $s$ and $d$:
\begin{enumerate}
\item $9\epsilon \le \alpha$,
\item $2.5\alpha+5.5\epsilon\le 1$,
\item $11\epsilon + 1\le 3\alpha$,
\end{enumerate}

\noindent If $T=(G,S,L)$ is a non-singular critical canvas with $c(S)\le 2$, then $d(T)\ge 3$.
\end{thm}

The proof of Theorem~\ref{StrongLinear} is given in Section~\ref{Proof}. 

\subsection{Deriving the Main Theorem}

We proceed to derive Theorem~\ref{Cylinder} from Theorem~\ref{StrongLinear} as follows. First we determine the appropriate $\epsilon, \alpha$ and $\gamma$.

\begin{theorem}\label{StrongLinear2}
If $G$ is a planar graph of girth at least five and $S$ is a subgraph of $G$ such that $G$ is $S$-critical for some $3$-list-assignment $L$ and $S$ has at most two components, then

$$\frac{1}{88}|V(G)\setminus V(S)| + \frac{3}{8}|E(S,G\setminus S)| \le 3|E(G)|-5|V(G)| + 5|V(S)|-3|E(S)| + 10(c(S)-c(G)).$$

Furthermore,

$$|V(G)| \le 177|V(S)|+528 \le 393|V(S)|.$$
\end{theorem}
\begin{proof}
Let $T=(G,S,L)$ and note that $T$ is a critical canvas. Let $\epsilon = 1/88$ and $\alpha=33\epsilon=3/8$. Note that inequality (1) of Theorem~\ref{StrongLinear} clearly holds. Moreover, $2.5 \alpha + 5.5 \epsilon = (82.5+5.5)\epsilon \le 1$ and hence inequality (2) holds.Moreover, $3\alpha= 99\epsilon \ge 11\epsilon+1$ and so inequality (3) holds. Thus $\epsilon$ and $\alpha$ satisfies the hypotheses of Theorem~\ref{StrongLinear}. The first formula now follow from Theorem~\ref{StrongLinear} and Proposition~\ref{d0} which give that $d(T)\ge 0$.

For the second formula, we note that by Euler's formula since $G$ is planar and has girth at least five that $3|E(G)|\le 5|V(G)|-10$. Meanwhile, $c(G)\ge 1$ and $c(S)\le 2$, so $c(S)-c(G)\le 1$. Finally note that $|E(S)|\ge |V(S)|-2$ since $S$ has at most two components. Thus the right side of the first formula is at most $-10 + 2|V(S)|+6 + 10 = 2|V(S)|+6$.  The left side is at least $\frac{1}{88}(|V(G)|-|V(S)|)$. The second formula now follows, where the last inequality holds as $|V(S)|\ge 1$ by Theorem~\ref{Thom0}.
\end{proof}

Theorem~\ref{Cylinder} would now be an immediate corollary of Theorem~\ref{StrongLinear2} with constant 393 in place of 99. However, we can do better. For example, in the case that $S$ is a facial cycle, we can derive a stronger bound as follows.

\begin{corollary}\label{StrongLinear3}
If $G$ is a plane graph of girth at least five and $C$ is a cycle of $G$ such that $G$ is $C$-critical for some $3$-list-assignment $L$, then

$$\frac{1}{88}|V(G\setminus C)| + \frac{3}{8}|E(C,G\setminus C)| + \sum_{f\in \F(G)} (|f|-5) \le 2|V(C)|-10.$$

Furthermore, if $C$ is facial, then $|V(G)| \le 89|V(C)|$.
\end{corollary}
\begin{proof}
By Euler's formula, $3|E(G)|-5|V(G)| = -10 - \sum_{f\in \F(G)} (|f|-5)$. Since $C$ is a cycle, $|V(C)|=|E(C)|$. Moreover, $c(G)\ge 1$ and $c(C)\le 1$. The first formula now follows from Theorem~\ref{StrongLinear2} with $C$ in place of $S$ as the right side of Theorem~\ref{StrongLinear2} is at most $2|V(C)|-10 - \sum_{f\in \F(G)} (|f|-5)$.

As for the second formula, if $C$ is facial, then as $G$ has girth at least five, we find that $\sum_{f\in F(G)}(|f|-5) \ge |C|-5$. Thus $\frac{1}{88}|V(G\setminus C)| \le |V(C)|-5$. Hence $|V(G)|\le 89|V(C)|$ as desired.
\end{proof}

We are now ready to prove Theorem~\ref{Cylinder}.

{\bf Proof of Theorem~\ref{Cylinder}.}
By Euler's formula, $3|E(G)|-5|V(G)| = -10 - \sum_{f\in \F(G)} (|f|-5)$. Since $C_1$ and $C_2$ are cycles, $|V(C_1)|=|E(C_1)|$ and $|V(C_2)|=|E(C_2)|$. Let $S=C_1\cup C_2$. Note that $c(S)\le 2$ while $c(G)\ge 1$. In addition, as $S=C_1\cup C_2$, $|V(S)|\le |E(S)|$ since if $S$ is connected, $S$ contains a cycle. Note that $G$ is $S$-critical by assumption. Since $C_1,C_2$ are facial cycles, $\sum_{f\in \F(G)} (|f|-5) \ge |V(C_1)|+|V(C_2)|-10$. Thus by Theorem~\ref{StrongLinear2}, $\frac{1}{88}|V(G\setminus S)| \le |V(S)|+10(c(S)-c(G))$. 

If $S$ is connected, then this is at most $|V(C_1)|+|V(C_2)|$ and we find that $|V(G)|\le 89 (|V(C_1)|+|V(C_2)|)$ as desired. So we may assume that $S$ is not connected, that is $C_1$ and $C_2$ are disjoint. As $G$ has girth at least five, $|V(C_1)|,|V(C_2)|\ge 5$. Thus $|V(S)|+10\le 2(|V(C_1)|+|V(C_2)|)$. Hence we find that $|V(G)|\le 177(|V(C_1)|+|V(C_2)|)$ as desired. \qed

\section{Proof of Theorem~\ref{StrongLinear}}\label{Proof}

We say a canvas $T_1=(G_1,S_1,L_1)$ is \emph{smaller} than a canvas $T_2=(G_2,S_2,L_2)$ if either

\begin{itemize}
\item $v(T_1) < v(T_2)$, or 
\item $v(T_1)=v(T_2)$ and $e(T_1) < e(T_2)$, or 
\item $v(T_1)=v(T_2)$, $e(T_1)=e(T_2)$ and $\sum_{v\in V(G_1)\setminus V(S_1)} |L(v)| < \sum_{v\in V(G_2)\setminus V(S_2)} |L(v)|$.
\end{itemize}

Let $T_0=(G_0,S_0,L_0)$ be a counterexample to Theorem~\ref{StrongLinear} such that every canvas $T$ smaller than $T_0$ that satisfies the assumptions of Theorem~\ref{StrongLinear} have $d(T)\ge 4-\gamma$.

We say a canvas $T$ is \emph{close} to $T_0$ if $T=T_0$ or $T$ is smaller than $T_0$ and $d(T_0)\ge d(T)-6\epsilon$.

Note that $3 > 4-3\alpha-\epsilon > 3-2\alpha$ by inequalities (1) and (2). Thus every critical canvas $T$ smaller than $T_0$ satisfies $d(T)\ge 3-2\alpha$. Finally it is useful to note one more inequality:

$$4. 2\alpha + 10\epsilon \le 1,$$

which follows from inequalities (1) and (2) since $1\ge 2.5\alpha + 5.5\epsilon \ge 2\alpha + 4.5\epsilon + 5.5\epsilon = 2\alpha + 10\epsilon$.

\subsection{Properties of Close Canvases}

For the remainder of the proof of Theorem~\ref{StrongLinear}, let $T=(G,S,L)$ be a critical canvas close to $T_0$ such that $c(S)\le 2$. We proceed to establish many properties of such a $T$. In particular that $T$ has none of the following: an edge not in $S$ with both ends in $S$, a vertex not in $S$ with at least two neighbors in $S$, a neighboring claw of $S$, a semi-neighboring $3$-path of $S$ or a semi-neighboring $5$-path of $S$. Finally we will also show that a $\le 2$-relaxation of $T_0$ is close to $T_0$. Hence if we can apply Lemma~\ref{Structure3} to $T_0$, the proof will be complete; our next claim shows that and more.

\begin{claim}\label{Dist9}
If $S$ has two components $S_1$ and $S_2$, then $d(S_1,S_2)\le 9$.
\end{claim}
\begin{proof}
Suppose not. Let $P$ be a shortest path from $S_1$ to $S_2$. By Corollary~\ref{Supercanvas}, $T/(S\cup P)$ is critical. As $T/(S\cup P)$ is smaller than $T_0$, $T/(S\cup P)$ is not a counterexample to Theorem~\ref{StrongLinear}. Hence, $d(T/(S\cup P))\ge 2$. Furthermore if $|E(P)|\ge 3$, then $d(T/(S\cup P))\ge 3$ since in that case $|V(G) \setminus V(S\cup P)|\ge 2$. 

Note that $e(T)-e(T/(S\cup P)) = |E(P)|$ and $v(T) - V(T/(S\cup P)) = |V(P)|-2$. Yet $S\cup P$ is connected where as $S$ is not. Finally, note that $q(T/(S\cup P))-q(T) \ge |V(P)|-4$ as the $|V(P)|-2$ internal vertices of $P$ have degree 3 in $G$ and so will count for at least one in $q(T/(S\cup P))$ while the first and last edges of $P$ which were counted in $q(T)$ will no longer count for $q(T/(S\cup P))$. Combining these observations we have 

$$d(T) \ge 10 + d(T/(S\cup P)) +3|E(P)|-(5+\epsilon)(|V(P)|-2) + (|V(P)|-4)\alpha.$$ 

\noindent Since $|E(P)|=|V(P)|-1$, 

$$d(T) \ge d(T/(S\cup P)) + 15 + \epsilon - 3\alpha -(2+\epsilon - \alpha)|E(P)|.$$ 

If $|E(P)|\le 2$, then as $d(T/(S\cup P))\ge 2$, we have that $d(T) \ge 13-\epsilon - \alpha $, but this at least $3+6\epsilon$ by inequality (2); hence $d(T_0)\ge 3$ as $T$ is close to $T_0$, a contradiction. So we may assume that $|E(P)|\ge 3$ and hence $d(T/(S\cup P))\ge 3$. Thus 

$$d(T) \ge 18 + \epsilon - 3\alpha - (2+\epsilon-\alpha)|E(P)|.$$ 

\noindent Note that $2+\epsilon-\alpha$ is positive by inequality (2) and hence the right side is minimized when $|E(P)|$ is maximized. Since $|E(P)|\le 8$, it follows that 

$$d(T) \ge 2-7\epsilon + 5\alpha.$$

\noindent Yet $5\alpha \ge 3\alpha + 2\alpha \ge (1+11\epsilon) + 18\epsilon = 1 = 29 \epsilon$ by inequalities (1) and (3). Thus we find that $d(T)\ge 3 + 29\epsilon$ and hence $d(T_0)\ge 3$, a contradiction.  
\end{proof}

So we may assume by Claim~\ref{Dist9} that either $c(S)=1$ or that if $S$ has two components $S_1,S_2$, then $d(S_1,S_2)\ge 9$. Hence by applying Lemma~\ref{Structure3} to $T_0$, there exists one of Lemma~\ref{Structure3}(1)-(5). We shall proceed to show that the existence of each of these yields a contradiction as described above. However first we will need some further claims about subcanvases of $T$.

\subsection{Chords and Neighbors of $S$}

\begin{claim}\label{D1}
If $T|H=(H,S,L)$ is a proper subcanvas of $T$ such that $S$ is a proper subgraph of $H$, then $d_T(T|H) < 1+\epsilon$. Further, if $|V(G)\setminus V(H)|\ge 2$, then $d_T(T|H) < 6\epsilon$. 
\end{claim}
\begin{proof}
Suppose not. By Proposition~\ref{surplussum2}, $d(T)\ge d_T(T|H) + d(T/H)$. By Corollary~\ref{Supercanvas}, $T/H$ is critical. As $T/H$ is smaller than $T_0$, $d(T/H) \ge 3-2\alpha$. Thus if $d_T(T|H)\ge 1+\epsilon$, then $d(T)\ge 4-2\alpha \ge 3+6\epsilon$ by inequality (4); hence $d(T_0)\ge 3$, a contradiction. Further, if $|V(G)\setminus V(H)|\ge 2$, then $T/H$ is non-singular and hence $d(T/H) \ge 3$. So if $d_T(T|H)\ge 6\epsilon$, then $d(T)\ge 3+6\epsilon$ and hence $d(T_0)\ge 3$, a contradiction. 
\end{proof}

Claim~\ref{D1} has the following useful corollaries, namely that Lemma~\ref{Structure3}(1) and Lemma~\ref{Structure3}(2) do not exist in $T$.

\begin{cor}\label{Chord}
There does not exist an edge in $E(G)\setminus E(S)$ with both ends in $S$.
\end{cor}
\begin{proof}
Suppose not. Let $H$ be the subgraph consisting of the union of $S$ and an edge in $E(G)\setminus E(S)$ with both ends in $S$. As $T$ is not a chord, $T|H$ is a proper subcanvas. Yet $d(T|H)\ge 3-2\alpha$ and hence $d(T|H)\ge 1+\epsilon$ by inequality (4). Since $d_T(T|H)\ge d(T|H)$, we find that $d_T(T|H)\ge 1+\epsilon$, contradicting Claim~\ref{D1}. 
\end{proof}

\begin{cor}\label{2Neighbors}
There does not exist a vertex $v\in V(G)\setminus V(S)$ with at least two neighbors in $S$.
\end{cor}
\begin{proof}
Suppose not. That is, there exists a vertex $v\in V(G)\setminus V(S)$ with two neighbors $u_1,u_2\in V(S)$. Let $H=S\cup \{u_1v, u_2v\}$. Note that $T|H$ is a proper subcanvas of $T$ as $v$ has degree two in $H$ but degree at least three in $G$. Yet $d(T|H)\ge 1-2\alpha-\epsilon > 6\epsilon$ by inequality (4). Thus by Claim~\ref{D1}, $|V(G)\setminus V(H)|\le 1$. Hence $v(T)=2$. But then $d(T)\ge 3(5) - (5+\epsilon)(2) - 4\alpha = 5-2\epsilon-4\alpha$ which is at least $3+6\epsilon$ by inequality (4); hence $d(T_0)\ge 3$, a contradiction.
\end{proof}

Hence by Corollary~\ref{2Neighbors}, $|N(v)\cap V(S)|=1$ for all $v\in N(S)\setminus V(S)$. Moreover, by Claims~\ref{Chord} and~\ref{2Neighbors}, $T$ is normal. Another useful corollary of Claim~\ref{D1} is the following claim which bounds the length of neighboring paths.

\begin{claim}\label{5Path}
There does not exist a neighboring $3$-path of $S$.
\end{claim}
\begin{proof}
Suppose not. Let $P=p_1p_2p_3p_4$ be a neighboring $3$-path of $S$. Let $H$ be the subgraph of $G$ induced by $V(P)\cup V(S)$. Note that $v(T|H) = 4$, $q(T|H)=4$ and $e(T|H) = 7$ as each vertex of $P$ has a unique neighbor in $S$ by Claim~\ref{2Neighbors}. Thus $d(T|H)\ge 3(7)-(5+\epsilon)(4)-4\alpha = 1-4\epsilon-4\alpha$. Yet $q_T(H,S)=\sum_{v\in V(H)\setminus V(S)}({\rm deg}_{G-H}(v)\ge 2$ given that $\deg{H}(p_1)=\deg_{H}(p_4)= 2$. Thus $d_T(T|H) \ge 1-2\alpha-4\epsilon$. This is at least $6\epsilon$ by inequality (4).

Thus by Claim~\ref{D1}, $|V(G)\setminus V(H)|\le 1$. As $p_1\not\sim p_4$ since $G$ has girth at least five, it follows that $|V(G)\setminus V(H)|=1$. Let $p_5\in V(G)\setminus V(H)$. It follows that $p_5$ is adjacent to $p_1$ and $p_4$ and exactly one vertex of $S$. It follows from Corollaries~\ref{Chord} and \ref{2Neighbors} that $G$ consists of $S$ and a $5$-cycle $p_1p_2p_3p_4p_5$ of vertices of degree three. Thus $d(T)\ge 3(10)-(5+\epsilon)5 -5\alpha = 5-5\alpha-5\epsilon$, which is at least $3+6\epsilon$ by inequality (2), a contradiction.
\end{proof}

The next claim shows that the components of $G[N(S)\setminus V(S)]$ have maximum degree two and hence are paths or cycles, that is Lemma~\ref{Structure3}(3) does not exist in $T$.
 
\begin{claim}\label{NoTree}
There does not exists a neighboring claw of $S$.
\end{claim}
\begin{proof}
Suppose not. Let $v\in N(S)\setminus V(S)$ and $v_1,v_2,v_3 \in N(v)\cap (N(S)\setminus V(S))$. Let $H=G[V(S)\cup \{v,v_1,v_2,v_3\}]$. It follows from Claim~\ref{2Neighbors} and the fact that $G$ has girth at least five that ${\rm deg}_{H}(v_i)=2$ for all $i\in\{1,2,3\}$. Hence $T|H$ is a proper subcanvas of $T$. Moreover since $G$ has girth at least five, $|V(G)\setminus V(H)|\ge 2$. 

Note that $v(T|H) = 4$ and $e(T|H)=7$ since each of $\{v,v_1,v_2,v_3\}$ has a unique neighbor in $S$ by Corollary~\ref{2Neighbors}. Moreover, $q(T|H)=4$ and hence $d(T|H)=3(7)-(5+\epsilon)(4)-4\alpha = 1 -4\epsilon-4\alpha$. Yet $q_T(H,S) = \sum_{v\in V(H)\setminus V(S)}({\rm deg}_{G-H}(v)\ge 3$ since $\deg_{H}(v_i)=2$ for all $i\in\{1,2,3\}$. Thus $d_T(T|H) \ge 1-4\epsilon-\alpha$ which is at least $6\epsilon$ by inequality (4), contradicting Claim~\ref{D1}.
\end{proof}

Thus one of Lemma~\ref{Structure3}(4) or Lemma~\ref{Structure3}(5) exists in $T_0$. That is, there exists a $k$-relaxation $T_0'=(G_0,S_0',L_0)$ of $T_0$ with $k\le 2$ such that there exists a semi-neighboring $3$-path of $S_0'$, or a semi-neighboring $5$-path of $S_0'$. As mentioned before, we will prove that such a $T_0'$ is close to $T_0$. We shall also prove that $T$ does not have a semi-neighboring $3$-path of $S$ or a semi-neighboring $5$-path of $S$. Combined these facts will complete the proof, but before we can do that we will need improved bounds for subcanvases of canvases smaller than or equal to $T_0$, which the next subsection provides.

\subsection{Proper Critical Subgraphs}

Here is a very useful claim.

\begin{claim}\label{ProperCrit}
Suppose $T_1=(G_1,S_1,L_1)$ is a normal critical canvas such that $c(S_1)\le 2$ and either $T_1=T_0$ or $T_1$ is smaller than $T_0$. If $G_1$ contains a proper subgraph $H_1$ that is $S_1$-critical with respect to some $3$-list-assignment $L_1'$, then $d(T_1)\ge 6-\alpha$. Furthermore, if $|E(G_1)\setminus E(H_1)|\ge 2$, then $d(T_1)\ge 6$.
\end{claim}
\begin{proof}
Suppose not. Let $T_1'= (H_1,S_1,L_1')$. By assumption, $T_1'$ is critical. Since $T_1$ is normal, $T_1'$ is non-singular. As $T_1'$ is smaller than $T$, $d(T_1')\ge 3$. Note that $d(T_1|H_1) = d(T_1')$ and hence $d(T_1|H_1)\ge 3$. As $H_1$ is proper, $T_1/H_1$ is critical by Corollary~\ref{Supercanvas}. By Proposition~\ref{surplussum2}, $d(T_1)\ge d_{T_1}(T_1|H_1)+d(T_1/H_1)$.

If $T_1/H_1$ is non-singular, then since $T_1/H_1$ is smaller than $T_0$, we find that $d(T_1/H_1)\ge 3$ and hence $d(T_1)\ge 6$, a contradiction. 

So we may suppose that $T_1/H_1$ is singular, that is $T_1/H_1$ is either a chord or tripod. First suppose $T_1/H_1$ is a tripod. Then $d(T_1/H_1)\ge 4-3\alpha-\epsilon$. Yet as there does not exist a vertex not in $S_1$ with at least three neighbors in $S_1$ as $T_1$ is normal, at least one of the edges in $E(G_1)\setminus E(H_1)$ does not have an end in $S_1$. Thus $q_{T_1}(H_1,S_1)\ge 1$. It follows that $d(T_1)\ge 7-2\alpha-\epsilon$, which is at least $6$ by inequality (2), a contradiction.

So we may suppose that $T_1/H_1$ is a chord. Hence $|E(G_1)\setminus E(H_1)|=1$ and so $d(T_1/H_1) = 3-2\alpha$. However, as $T_1$ is normal, $q_{T_1}(H_1,S_1)\ge 1$. Hence $d(T_1)\ge 3 + 3-2\alpha + \alpha = 6-\alpha$, a contradiction. 
\end{proof}

\begin{cor}\label{NoProperCrit}
There does not exist a proper subgraph of $G$ that is $S$-critical with respect to some list assignment $L'$.
\end{cor}
\begin{proof}
Suppose not. Let $H$ be a proper subgraph of $G$ that is $S$-critical with respect to $L'$. By Corollaries~\ref{Chord} and~\ref{2Neighbors}, $T$ is normal. By Claim~\ref{ProperCrit} with $T_1=T$, $G_1=G$, $H_1=H$ and $L_1=L'$, we find that $d(T)\ge 6-\alpha$ which is at least $3+6\epsilon$ by inequality (2). As $T$ is close to $T_0$, we have that $d(T_0)\ge 3$, a contradiction.
\end{proof}

There exists a proper coloring $\phi$ of $S$ that does not extend to $G$ as $G$ is $S$-critical. Our next claim proves that $G$ is critical with respect to any such coloring. Actually we can prove more, but we need the following definition: we say a $3$-list-assignment $L'$ is \emph{nice} for $T$ if $L'(v)\subseteq L(v)$ for all $v\in V(G)$, $L'(v)=L(v)$ for all $v\in V(S)$ and $|L'(v)|=3$ for all $v\in V(G)\setminus V(S)$.

\begin{cor}\label{Precolored}
Let $L'$ be a nice $3$-list-assignment for $T$. If $\phi$ is an $L$-coloring of $S$ that does not extend to an $L$-coloring of $G$, then $\phi$ extends to an $L'$-coloring of every proper subgraph of $G$ containing $S$.
\end{cor}
\begin{proof}
Suppose not. That is, there exists a proper subgraph $H$ of $G$ such that $H$ contains $S$ and $\phi$ does not extend to an $L'$-coloring $H$. But then $H$ contains a subgraph $H'$ that is $S$-critical with respect to $L'$. As $H$ is a proper subgraph of $G$, $H'$ is also a proper subgraph of $G$, contradicting Corollary~\ref{NoProperCrit}.
\end{proof}

For the rest of the proof of Theorem~\ref{StrongLinear}, we fix an $L$-coloring $\phi$ of $S$ which does not extend to $G$ and we fix a nice $3$-list-assignment $L'$ for $T$. Note that $\phi$ does not extend to an $L'$-coloring of $G$ (and hence an $L$-coloring of $G$) as $T$ is critical. For $v\not\in V(S)$, we let 

$$A(v):=L'(v)\setminus \{\phi(u): u\in N(v)\cap V(S)\}.$$ 

Note that for every edge $e=uv$ with $u,v\not\in V(S)$, $L'(u)\cap L'(v)\ne \emptyset$ by Corollary~\ref{Precolored}. Similarly, for every vertex $v\not\in V(S)$, $|A(v)|=|L'(v)|-|N(v)\cap V(S)|=3-|N(v)\cap V(S)|$. This follows since by Corollary~\ref{Precolored}, $\phi(u)\in L'(v)$ for all $u\in N(v)\cap V(S)$ and $\phi(u)\ne \phi(w)$ for all $u\neq w \in N(v)\cap V(S)$.  

Here is the application of Claim~\ref{ProperCrit} that we repeatedly use for reductions.

\begin{claim}\label{Reduction}
Let $H$ be a subgraph of $G$ containing $S$ such that $T/H$ is normal. If there exists a proper subgraph $G'$ of $G$ such that $H\subseteq G'$ and there exists an $L'$-coloring of $H$ that does not extend to an $L'$-coloring of $G'$, then $d_T(T|H) < -3 + \alpha + 6\epsilon$. Further if $|E(G)\setminus E(G')|\ge 2$, then $d_T(T|H) < -3 + 6\epsilon$.
\end{claim}
\begin{proof}
Suppose not. As there exists an $L'$-coloring of $H$ that does not extend to an $L'$-coloring of $G'$, there exists a subgraph $G''$ of $G'$ that is $H$-critical with respect to $L'$. Apply Claim~\ref{ProperCrit} with $T_1 = (G,H,L) = T/H$, $H_1=G''$ and $L_1'=L'$. Thus $d(T/H) \ge 6-\alpha$ and furthermore if $|E(G)\setminus E(G')|\ge 2$, then $d(T/H)\ge 6$. 

By Proposition~\ref{surplussum2}, $d(T)\ge d_T(T|H)+d(T/H)$. As $d(T) < 3+6\epsilon$, we find that $d_T(T|H) < 3 +6\epsilon - d(T/H)$. Hence $d(T/H) < 3 + 6\epsilon - (6-\alpha) = -3 + \alpha + 6\epsilon$. Furthermore if $|E(G)\setminus E(G')|\ge 2$, then $d_T(T|H) < 3 + 6\epsilon - (6) = -3 + 6\epsilon$, in either case a contradiction to the bounds above.
\end{proof}

In what remains of the the proof of Theorem~\ref{StrongLinear}, we will invoke Claim~\ref{Reduction} to appropriately chosen $G'$ and $H$ to show that $T$ does not have a semi-neighboring $3$-path or semi-neighboring $5$-path. Before we do that, we need a preliminary claim whose proof also relies on Claim~\ref{Reduction}. For the coloring arguments, it is useful to note the following claim.

\begin{claim}\label{ColorSubset}
If $v\in N(S)$ and $u\in N(v)\setminus V(S)$, then $A(v)\subseteq A(u)$.
\end{claim}
\begin{proof}
Suppose not. Let $H=G[V(C)\cup v]$. Note that $d_T(T|H) \ge 3-(5+\epsilon) - \alpha + 2\alpha = -2-\epsilon+\alpha$ which is at least $-3+\alpha+6\epsilon$ as $7\epsilon \le 1$ by inequalities (1) and (2). Furthermore, it follows from Claims~\ref{Chord} and \ref{2Neighbors} that there does not exists an edge in $E(G)\setminus E(H)$ with both ends in $V(H)$ or a vertex in $V(G)\setminus V(H)$ with three neighbors in $V(H)$ and hence $T/H$ is normal.

Let $\phi(v)\in A(v)\setminus A(u)$. Now $\phi$ is an $L'$-coloring of $H$. Let $G'=G\setminus \{vu\}$. Thus $G'$ is a proper subgraph of $G$ such that $H\subseteq G'$. Yet $\phi$ does not extend to an $L'$-coloring of $G'$ as otherwise $\phi$ extends to an $L'$-coloring of $G$ (and hence an $L$-coloring of $G$), contradicting that $T$ is critical. So we may assume that $\phi$ does not extend to an $L'$-coloring of $G'$. Thus by Claim~\ref{Reduction}, $d_T(T|H) < -3+\alpha+6\epsilon$, a contradiction.
\end{proof}

\subsection{Neighboring Paths}

By Claim~\ref{5Path} only neighboring $1$-paths and neighboring $2$-paths of $S$ may exist. We cannot directly obtain a contradiction by their existence. Hence we also need more information on the degrees of vertices in neighboring $1$-paths which the following claim provides.

\begin{claim}\label{NotBothDeg3}
If $P=p_1p_2$ is a neighboring $1$-path of $S$ such that $P$ is a component of $N(S)\setminus S$ of size two, then either ${\rm deg}(p_1)\ge 4$ or ${\rm deg}(p_2)\ge 4$.
\end{claim}
\begin{proof}
Suppose not. Hence ${\rm deg}(p_1)={\rm deg}(p_2)=3$. By Claim~\ref{ColorSubset}, $A(p_1)=A(p_2)$. Let $u_1,u_2$ be such that $N(p_1)\setminus (V(S)\cup \{p_2\})=\{u_1\}$ and $N(p_2)\setminus (V(S)\cup \{p_1\})=\{u_2\}$. Note that $u_1,u_2\not\in V(S)$ by Claim~\ref{2Neighbors}. Furthermore, $u_1\ne u_2$ as $G$ has girth at least five. As $L'$ is nice, $|L'(u_1)|=|L'(u_2)|=3$. Yet, $u_1, u_2\not\in N(S)$ by assumption and hence $|A(u_1)|=|A(u_2)|=3$. We may assume without loss of generality that ${\rm deg(u_1)} \ge {\rm deg}(u_2)$.

Let $H=G[V(C)\cup \{u_1,p_1,p_2\}]$. Note that $e(T|H) = e(T)-4$, $v(T|H)=v(T)-3$ and $q(T|H)=2$. Yet $q_T(H,S) = {\rm deg}(u_1)$ as ${\rm deg}_{G-H}(u_1) = {\rm deg}(u_1)-1$ and ${\rm deg}_{G-H}(p_2)=1$. Combining we find that
$$d_T(T|H) = 3(4)-(5+\epsilon)(3) - 2\alpha + {\rm deg}(u_1)\alpha = -3 + ({\rm deg}(u_1)-2)\alpha-3\epsilon,$$ 
which is at least $-3+\alpha - 3\epsilon$ as $\deg(u_1)\ge 3$. By inequality (1), $-3+\alpha-3\epsilon \ge -3 + 6\epsilon$. Moreover, this is at least $-3+\alpha+6\epsilon$ if ${\rm deg}(u_1)\ge 4$ as $\alpha\ge 9\epsilon$ by inequality (1). Furthermore, it follows from Corollaries~\ref{Chord} and \ref{2Neighbors} and the fact that $G$ has girth at least five that there does not exist an edge in $E(G)\setminus E(H)$ with both ends in $V(H)$ or a vertex in $V(G)\setminus V(H)$ with three neighbors in $V(H)$. Hence $T/H$ is normal.

Let $\phi(u_1)\in A(u_1)\setminus A(p_1)$. Then let $\phi(p_2)\in A(p_2)$ and $\phi(p_1)\in A(p_1)\setminus \{\phi(p_2)\}$. 

First suppose ${\rm deg}(u_1)\ge 4$. Let $G'=G\setminus \{p_2u_2\}$. Now if $\phi$ extends to an $L'$-coloring of $G'$, then there exists an $L'$-coloring of $G$ by recoloring $p_2$ different from $u_2$ and then recoloring $p_1$ different from $p_2$, contradicting that $T$ is critical. Thus by Claim~\ref{Reduction}, $d_T(T|H) < -3+\alpha+6\epsilon$, a contradiction to the bound given above.

So we may suppose that ${\rm deg}(u_1)=3$. Thus ${\rm deg}(u_2)=3$ since we assumed that $\deg(u_1)\ge \deg(u_2)$. Let $G'=G\setminus \{u_2\}$. Now if $\phi$ extends to an $L'$-coloring of $G'$, then there exists an $L'$-coloring of $G$ by coloring $u_2$ and then recoloring $p_2,p_1$ in that order, contradicting that $T$ is critical. So we may assume that $\phi$ does not extend to an $L'$-coloring of $G'$. Yet $|E(G)\setminus E(G'')|\ge 2$. Thus by Claim~\ref{Reduction}, $d_T(T|H) < -3+6\epsilon$, a contradiction to the bound given above.
\end{proof}

We also need more information about the degrees of vertices in neighboring $2$-paths as follows.

\begin{claim}\label{NotAllDeg3}
If $P=p_1p_2p_3$ is a neighboring $2$-path of $S$, then ${\rm deg}(p_1)+{\rm deg}(p_2)+\deg(p_3)\ge 10$.
\end{claim}
\begin{proof}
Suppose not. As $G$ is $S$-critical, ${\rm deg}(p_i)\ge 3$ for all $i\in\{1,2,3\}$. It now follows that ${\rm deg}(p_1)={\rm deg}(p_2) = {\rm deg}(p_3) = 3$. Let $u_1$ be the unique neighbor of $p_1$ in $V(G)\setminus (V(S)\cup \{p_2\})$. Similarly let $u_3$ be the unique neighbor of $p_3$ in $V(G)\setminus (V(S)\cup \{p_2\})$. Note that $u_1,u_3\not\in V(S)$ by Claim~\ref{2Neighbors}. Furthermore, $u_1\ne u_3$ as $G$ has girth at least five. As $L'$ is nice, $|L'(u_1)|=|L'(u_3)|=3$. Yet, $u_1, u_3\not\in N(S)$ by Claim~\ref{5Path} and hence $|A(u_1)|=|A(u_3)|=3$.

Let $H=G[V(C)\cup \{u_1,p_1,p_2,p_3\}]$. Note that $e(T|H) = e(T)-6$, $v(T|H)=v(T)-4$ and $q(T|H)=3$. Yet $q_T(H,S) = {\rm deg}(u_1)$ as ${\rm deg}_{G-H}(u_1) = {\rm deg}(u_1)-1$ and ${\rm deg}_{G-H}(p_3)=1$. Combining we find that
$$d_T(T|H) = 3(6)-(5+\epsilon)(4) - 3\alpha + {\rm deg}(u_1)\alpha = -2 + ({\rm deg}(u_1)-3)\alpha-4\epsilon,$$ 
which is at least $-2 - 4\epsilon$ as $\deg(u_1)\ge 3$. By inequality (4), $-2-4\epsilon \ge -3 \alpha + 6\epsilon$. Moreover, it follows from Corollaries~\ref{Chord} and \ref{2Neighbors} and the fact that $G$ has girth at least five that there does not exist an edge in $E(G)\setminus E(H)$ with both ends in $V(H)$ or a vertex in $V(G)\setminus V(H)$ with three neighbors in $V(H)$. Hence $T/H$ is normal.

Let $\phi(u_1)\in A(u_1)\setminus A(p_1)$. Then let $\phi(p_3)\in A(p_3)$, $\phi(p_2)\in A(p_2)\setminus \{\phi(p_3)\}$ and $\phi(p_1)\in A(p_1)\setminus \{\phi(p_2)\}$. 

Let $G'=G\setminus \{p_3u_3\}$. Now if $\phi$ extends to an $L'$-coloring of $G'$, then there exists an $L'$-coloring of $G$ by recoloring $p_3$ different from $u_3$ and then recoloring $p_2,p_1$ in that order if necessary, contradicting that $T$ is critical. So we may assume that $\phi$ does not extend to an $L'$-coloring of $G'$. Thus by Claim~\ref{Reduction}, $d_T(T|H) < -3+\alpha+6\epsilon$, a contradiction to the bound given above.
\end{proof}

Claim~\ref{NotAllDeg3} is enough to show that a relaxation of $T$ has a value of $d$ close to that of $T$ as follows.

\begin{cor}\label{Relax1}
If $T'$ is a relaxation of $T$, then $d(T)\ge d(T')-3\epsilon$.
\end{cor}
\begin{proof}
By the definition of relaxation, there exists a neighboring $2$-path $P=p_1p_2p_3$ of $S$ such that for each $i\in\{1,2,3\}$ there is a unique neighbor of $p_i$ in $S$, call it $u_i$, such that if $H=G\cup P \cup \{p_1u_1,p_2u_2,p_3u_3\}$, then $T'=T/H$. 

Note that $v(T|H)=3$, $e(T|H)=5$ and $q(T|H)=3$. Hence $d(T|H) = 3(5)-(5+\epsilon)3 - 3\alpha = -3\epsilon - 3\alpha$. By Claim~\ref{NotAllDeg3}, ${\rm deg}_G(p_1)+{\rm deg}_G(p_2)+\deg_G(p_3)\ge 10$. Yet ${\rm deg}_H(p_1)+{\rm deg}_H(p_2)+\deg_H(p_3)=7$. Thus $q_T(H,S) \ge 3$. So $d_T(T|H) = d(T|H) +\alpha q_T(H,S) \ge -3\epsilon$. By Proposition~\ref{surplussum2}, 
$$d(T)\ge d(T/H) + d_T(T|H) \ge d(T')-3\epsilon,$$
as desired.
\end{proof}

From Corollary~\ref{Relax1}, we have that a $\le 2$-relaxation of $T_0$ is close to $T_0$ as follows.

\begin{cor}\label{Relax2}
If $T_1$ is a $k$-relaxation of $T_0$ with $k\le 2$, then $T_1$ is close to $T_0$.
\end{cor}
\begin{proof}
We proceed by induction on $k$. If $k=0$, then $T_1=T_0$ and hence is close to $T_0$ as desired. If $k=1$, then as $T_0$ is close to $T_0$, it follows from Corollary~\ref{Relax1} that $d(T_0)\ge d(T_1)-3\epsilon$ and hence $T_1$ is close to $T_0$ as desired. So we may assume that $k=2$. But then $T_1$ is a $1$-relaxation $T_1'$ of $T_0$. By induction, $T_1'$ is close to $T_0$. So by Corollary~\ref{Relax1}, $d(T_1')\ge d(T_1)-3\epsilon$. Yet $d(T_0)\ge d(T_1')-3\epsilon$ by Corollary~\ref{Relax1} and hence $d(T_0)\ge d(T_1)-6\epsilon$ and so $T_1$ is close to $T_0$ as desired.
\end{proof}

\subsection{Semi-Neighboring Paths}

 We now proceed to show that there does not exist a semi-neighboring $3$-path or semi neighoring $5$-path of $S$.

\begin{claim}\label{Semi3}
There does not exist a semi-neighboring $3$-path of $S$.
\end{claim}
\begin{proof}
Suppose not. Let $P=p_1\ldots p_4$ be a semi-neighboring $3$-path of $C_1$. By Claim~\ref{2Neighbors}, $|N(p_i)\cap V(C)| = 1$ for $i\in \{1,2,4\}$ and hence $|A(p_i)|=2$ for $i\in\{1,2,4\}$. As there does not exist a neighboring $3$-path of $C$ by Claim~\ref{5Path}, we find that $p_3\not\in N(C)$ and hence $|A(p_3)|=3$. By Claim~\ref{ColorSubset}, we find that $A(p_1)=A(p_2)\subset A(p_3)$. 

Let $R= (N(p_1)\cup N(p_2))\cap (N(S) \setminus \{p_1,p_2\})$. By Claims~\ref{5Path} and~\ref{NoTree}, $|R|\le 1$. Let $H=G[V(S)\cup \{p_1,p_2,p_3\} \cup R]$.

\begin{subclaim}\label{DFirst}
$d_T(T|H) \ge -3+\alpha+6\epsilon$.
\end{subclaim}
\begin{proof}
First suppose $|R|=1$ and let $x\in R$. Hence $e(T|H)=e(T)-6$, $v(T|H)=v(T)-4$ and $q(T|H)=3$. Note that ${\rm deg}_{G-H}(p_3)\ge 2$. Yet ${\rm deg}_G(x) + {\rm deg}_G(p_1) + {\rm deg}_G(p_2) \ge 10$ by Claim~\ref{NotAllDeg3} while ${\rm deg}_H(x) + {\rm deg}_H(p_1) + {\rm deg}_H(p_2) = 8$. Hence ${\rm deg}_{G-H}(x) + {\rm deg}_{G-H}(p_1) + {\rm deg}_{G-H}(p_2) \ge 2$. Thus $q_T(H,S)\ge 4$. Combining, we find that 

$$d_T(T|H)\ge 3(6)-(5+\epsilon)(4) - 3\alpha + 4\alpha = -2-3\epsilon+\alpha,$$ 

\noindent which is at least $-3+\alpha+6\epsilon$ as  claimed since $9\epsilon\le 1$ by inequalities (1) and (2). 

So we may assume that $|R|=0$. Hence $e(T|H)=e(T)-4$, $v(T|H)=v(T)-3$ and $q(T|H)=2$. Note that ${\rm deg}_{G-H}(p_3)\ge 2$. Yet ${\rm deg}_G(p_1) + {\rm deg}_G(p_2) \ge 7$ by Claim~\ref{NotBothDeg3} while ${\rm deg}_H(p_1) + {\rm deg}_H(p_2) = 5$. Hence ${\rm deg}_{G-H}(p_1) + {\rm deg}_{G-H}(p_2) \ge 2$. Thus $q_T(H,S)\ge 4$. Combining, we find that 

$$d_T(T|H) \ge 3(4)-(5+\epsilon)(3)-2\alpha + 4\alpha= -3 + 2\alpha-3\epsilon,$$

\noindent which is at least $-3+\alpha+6\epsilon$ as claimed since $9\epsilon \le \alpha$ by inequality (1).
\end{proof}

\begin{subclaim}\label{No3First}
$T/H$ is normal.
\end{subclaim}
\begin{proof}
Suppose not. There does not exist an edge with both endpoints in $V(H)$ by Claims~\ref{Chord}, \ref{2Neighbors} and \ref{NoTree} and the fact that $G$ has girth at least five. So we may assume there there exists a vertex $v\in V(G)\setminus V(H)$ with at least three neighbors in $V(H)$. By Claim~\ref{2Neighbors}, $v$ would have to have at least two neighbors in $V(H)\setminus V(S)$. As $G$ has girth at least five, it follows that these two neighbors have to be $p_3$ and a vertex $x$ in $R$. But then $v$ has another neighbor in $V(H)$ and hence $v\in N(S)$. Moreover since $G$ has girth at least five, $x$ is adjacent to $p_1$. Thus $p_2p_1xp_4$ is a neighboring $3$-path of $S$, contradicting Claim~\ref{5Path}.
\end{proof}

Let $\phi(p_3)\in A(p_3)\setminus A(p_4)$. Then extend $\phi$ to an $L'$-coloring of $H$ by coloring $p_2,p_1$ and $R$ (if it exists) in that order. Let $G'=G\setminus \{p_3p_4\}$. Now if $\phi$ extends to an $L'$-coloring of $G'$, then there exists an $L'$-coloring of $G$, contradicting that $T$ is critical. So we may assume that $\phi$ does not extend to an $L'$-coloring of $G'$. Given Subclaim~\ref{No3First}, it now follows from Claim~\ref{Reduction} that $d_T(T|H) < -3+\alpha+6\epsilon$, contradicting Subclaim~\ref{DFirst}.
\end{proof}

\begin{claim}\label{Semi5}
There does not exist a semi-neighboring $5$-path of $S$.
\end{claim}
\begin{proof}
Suppose not. Then there exists a semi-neighboring $5$-path $P=p_1\ldots p_6$ where $p_1,p_2,p_5,p_6\in N(C)$. We may assume without loss of generality that $\deg(p_3)\ge \deg(p_4)$. By Claim~\ref{ColorSubset}, $A(p_1)=A(p_2)\subset A(p_3)$ and $A(p_6)=A(p_5)\subset A(p_4)$.

Let $U_1=(N(p_1)\cup N(p_2))\cap (N(S)\setminus (V(S)\cup\{p_1,p_2\}))$ and similarly let $U_2=(N(p_5)\cup N(p_6))\cap (N(C)\setminus (V(S)\cup \{p_5,p_6\}))$. If $U_1\ne \emptyset$, then it follows from Claims~\ref{5Path} and~\ref{NoTree} that $|U_1|=1$.  Similarly if $U_2\ne \emptyset$, then $|U_2|=1$. 

Let $S_1=\{p_1,p_2\}\cup U_1$ and let $S_2=\{p_5,p_6\}\cup U_2$. By Claim~\ref{5Path}, $S_1\cap S_2 = \emptyset$ and there does not exist an edge with one end in $S_1$ and the other end in $S_2$. By Claim~\ref{Semi3}, there does not exist a vertex in $V(G)\setminus V(S)$ with a neighbor in both $S_1$ and $S_2$. Similarly by Claim~\ref{Semi3}, $p_3\not\in N(S_2)$ and $p_4\not\in N(S_1)$. Hence as $G$ has girth at least five, the graph induced by $G$ on $S_1\cup S_2\cup\{p_3,p_4\}$ has precisely $|S_1|+|S_2|+1 = 5 + |U_1|+|U_2|$ edges. 

Let $H = V(S) \cup S_1\cup S_2 \cup \{p_3\}$.

\begin{subclaim}\label{D}
$d_T(T|H) \ge -4 -5\epsilon + \deg(p_3)\alpha$.
\end{subclaim}
\begin{proof}
First note that $e(T|H) = 7 + 2|U_1|+2|U_2|$ and $v(T|H)=5+|U_1|+|U_2|$. Hence ${\rm def}(T|H)= 3e(T|H) - 5v(T|H) = 3(7+2|U_1|+2|U_2|)-5(5+|U_1|+|U_2|) = -4 + |U_1|+|U_2|$. Moreover, $q(T|H)=|S_1|+|S_2| = 4+|U_1|+|U_2|$ and hence $s(T|H)=(4+|U_1|+|U_2|)\alpha + (5+|U_1|+|U_2|)\epsilon$. Thus 

$$d(T|H) = -4 - 4\alpha - 5\epsilon + (|U_1|+|U_2|)(1-\alpha -\epsilon).$$ 

Let $B_1 = \sum_{v\in S_1} \deg_{G-H}(v)$. Note that $B_1\ge 1$. Furthermore, if $|U_1|=0$, then $\deg_G(p_1)+\deg_G(p_2)\ge 7$ by Claim~\ref{NotBothDeg3}. Hence if $|U_1|=0$, then $B_1\ge 2$. Thus $B_1\ge 2 - |U_1|$. Similarly let $B_2=\sum_{v\in S_2} \deg_{G-H}(v)$. Note that $B_2\ge 2$. Furthermore, if $|U_2|=0$, then $\deg(p_5)+\deg(p_6)\ge 7$ by Claim~\ref{NotBothDeg3}. Hence if $|U_2|=0$, then $B_2=3$. Thus $B_2\ge 3-|U_2|$. Also note that $\deg_{G-H}(p_3)=\deg_G(p_3)-1$. Hence $q_T(H,S) \ge B_1+B_2+\deg(p_3)-1$. So $d_T(T|H) \ge d(T|H) + (B_1+B_2+\deg(p_3)-1)\alpha$. 

Using the bounds above, we find that 
\begin{align*}
d_T(T|H) &\ge -4 - 4\alpha - 5\epsilon + (|U_1|+|U_2|)(1-\alpha-\epsilon) + (4+\deg(p_3) - |U_1|-|U_2|)\alpha \\
&= -4 - 5\epsilon +\deg(p_3)\alpha + (|U_1|+|U_2|)(1-2\alpha-\epsilon).
\end{align*}

\noindent Since $2\alpha+\epsilon \le 1$ by inequality (2), $d_T(T|H) \ge -4-5\epsilon +\deg(p_3)\alpha$ as claimed.
\end{proof}

\begin{subclaim}\label{No3}
$T/H$ is normal.
\end{subclaim}
\begin{proof}
Suppose not. Since $H$ is an induced subgraph of $G$ there does not exist an edge of $E(G)\setminus E(H)$ with both ends in $V(H)$. So we may assume there exists a vertex $v\in V(G)\setminus V(H)$ with at least three neighbors in $V(H)$. If $v\in N(S)$, then by Claim~\ref{2Neighbors}, $v$ would have to have at least two neighbors in $V(H)\setminus V(S)$. Hence $v$ has at least one neighbor in either $S_1$ or $S_2$, contradicting either Claim~\ref{5Path} or Claim~\ref{NoTree}. So we may assume that $v\not\in N(S)$. But then as $G$ has girth at least five, $v$ has a neighbor in both $S_1$ and $S_2$, contradicting Claim~\ref{Semi3}.
\end{proof}

If $\deg(p_4)\ge 4$, let $G'=G\setminus \{p_3p_4\}$ and if $\deg(p_4)=3$, let $G'=G\setminus p_4$. If $A(p_3)\cap A(p_5)\ne\emptyset$, let $\phi(p_3)= \phi(p_5)\in A(p_3)\cap A(p_5)$. If $A(p_3)\cap A(p_5)=\emptyset$, then $A(p_3)\setminus A(p_4)\ne \emptyset$ and in that case let $\phi(p_3)\in A(p_3)\setminus A(p_4)$ and $\phi(p_5)\in A(p_5)$. In either case, then extend $\phi$ to the rest of $V(H)$ by coloring $p_2,p_1,u_1$ and $p_6,u_2$ in that order. Now if $\phi$ extends to an $L'$-coloring of $G'$, then $\phi$ extends to an $L'$-coloring of $G$, contradicting that $T$ is critical. So we may assume that $\phi$ does not extend to an $L'$-coloring of $G'$. 

Given Subclaim~\ref{No3}, it follows from Claim~\ref{Reduction} that $d_T(T|H) < -3+\alpha+6\epsilon$, and further if $\deg(p_4)=3$, then $d_T(T|H) < -3+6\epsilon$. If $\deg(p_4)\ge 4$, then $d_T(T|H) \ge -4 -5\epsilon + 4\alpha$ by Subclaim~\ref{D}. Since $3\alpha \ge 1 + 11\epsilon$ by inequality (3), $d_T(T|H) \ge - 3 + \alpha+6\epsilon$, a contradiction. So we may assume that $\deg(p_4)=3$ and hence $d_T(T|H)\ge -4 - 5\epsilon +3\alpha$ by Subclaim~\ref{D}. Since $3\alpha \ge 1 + 11\epsilon$ by inequality (3), $d_T(T|H)\ge -3+6\epsilon$, a contradiction.
\end{proof}

We now finish the proof by applying Lemma~\ref{Structure3} to $T_0$. If Lemma~\ref{Structure3}(1) holds, then as $T_0$ is close to itself, this contradicts Claim~\ref{Chord}. If Lemma~\ref{Structure3}(2) holds, then as $T_0$ is close to itself, this contradicts Claim~\ref{2Neighbors}. If Lemma~\ref{Structure3}(3) holds, then as $T_0$ is close to itself, this contradicts Claim~\ref{NoTree}. If Lemma~\ref{Structure3}(4) holds, then there exists a $\le 2$-relaxation $T_0'=(G_0,S_0',L_0)$ of $T_0$ such that there exists a semi-neighboring $3$-path of $S_0'$. But then $T_0'$ is close to $T_0$ by Corollary~\ref{Relax2} and so by Claim~\ref{Semi3}, there does not exist a semi-neighboring $3$-path of $S_0'$, a contradiction. 

So we may assume that Lemma~\ref{Structure3}(5) holds. That is, there exists a $\le 2$-relaxation $T_0'=(G_0,S_0',L_0)$ of $T_0$ such that there exists a semi-neighboring $5$-path of $S_0'$. But then $T_0'$ is close to $T_0$ by Corollary~\ref{Relax2} and so by Claim~\ref{Semi5}, there does not exist a semi-neighboring $5$-path of $S_0'$, a contradiction. This concludes the proof of Theorem~\ref{StrongLinear}.

\section{Exponentially Many Extensions of a Precoloring of a Cycle}\label{Exp}

In this section, we provide as promised the proof that the family of exponentially critical graphs is hyperbolic thereby completing the proof of Theorem~\ref{ExpSurface2} (see Theorem~\ref{ExpHyper}). Before we do that, we need to recall a number of defintions from~\cite{PostleThomas}.

\begin{definition}[Definition 3.1 in~\cite{PostleThomas}]
A \emph{ring} is a cycle or a complete graph on one or two vertices. A \emph{graph with
rings} is a pair $(G,\R)$, where $G$ is a graph and $\R$ is a set of vertex-disjoint rings in $G$. 
\end{definition}

\begin{definition}[Definition 3.2 in~\cite{PostleThomas}]
We say that a graph $G$ with rings $\R$ is \emph{embedded in a surface $\Sigma$} if the
underlying graph $G$ is embedded in $\Sigma$ in such a way that for every ring $R \in \R$ there exists a
component $\Gamma$ of the boundary of $\Sigma$ such that $R$ is embedded in $\Gamma$, no other vertex or edge of $G$ is embedded
in $\Gamma$, and every component of the boundary of $\Sigma$ includes some ring of $G$. 
\end{definition}

Now let us state the formal definition of a hyperbolic family. 

\begin{definition}[Definition 5.1 in~\cite{PostleThomas}]
Let $\F$ be a family of non-null embedded graphs with rings. We say that $\F$ is \emph{hyperbolic} if there exists a constant $c>0$ such that if $G\in\F$ is a graph with rings that is embedded in a surface $\Sigma$, then for every closed curve $\gamma: \mathbb{S}^1 \rightarrow \Sigma$ that bounds an open disk $\Delta$ and intersects $G$ only in vertices, if $\Delta$ includes a vertex of $G$, then the number of vertices of $G$ in $\Delta$ is at most $c(|\{x\in \mathbb{S}^1: \gamma(x)\in V(G)\}|-1)$. We say that $c$ is a \emph{Cheeger constant} of $\F$.
\end{definition}

Let us also state the definition of strongly hyperbolic for those who are interested though we do not need it.

\begin{definition}[Definition 7.1 in~\cite{PostleThomas}]
Let $\F$ be a hyperbolic family of embedded graphs with rings, let $c$ be a Cheeger constant for $\F$, and let $d := \lceil 3(2c + 1) \log_2(8c + 4)\rceil$. We say that $\F$ is \emph{strongly
hyperbolic} if there exists a constant $c_2$ such that for every $G \in \F$ embedded in a surface $\Sigma$ with rings and for every two disjoint cycles $C_1,C_2$ of length at most $2d$ in $G$, if there
exists a cylinder $\Lambda \subseteq \Sigma$ with boundary components $C_1$ and $C_2$, then $\Lambda$ includes at most $c_2$ vertices of $G$. We say that $c_2$ is a \emph{strong hyperbolic constant} for $\F$.
\end{definition}

Thomassen~\cite{ThomExp} proved the following in~\cite{ThomExp}.

\begin{theorem}\label{ThomExp}
If $G$ is planar graph of girth at least five and $L$ is a $3$-list-assignment of $V(G)$, then $G$ has at least $2^{|V(G)|/10000}$ distinct $L$-colorings.
\end{theorem}

In fact, Thomassen proved a stronger result as follows.

\begin{theorem}\label{ThomExp2}[cf. Theorem 4.3 in~\cite{ThomExp}]
Let $T=(G,S,L)$ be a canvas such that $S$ is path on at most two vertices. If $\phi$ is an $L$-coloring of $S$, then $\phi$ extends to at least $2^{|V(G)|/10000}$ distinct $L$-colorings of $G$.
\end{theorem}

We prove the following generalization of Theorem~\ref{ThomExp} about the number of extensions of a coloring of a cycle, which we define as follows.

\begin{definition}
If $T=(G,S,L)$ is a canvas and $\phi$ is an $L$-coloring of $S$, then we let $E_T(\phi)$ denote the number of \emph{extensions} of $\phi$ to $G$, that is the number of distinct $L$-colorings of $G$ whose restriction on $S$ is equal to $\phi$.
\end{definition}

\begin{lemma}\label{ExpManyDisc}
If $T=(G,S,L)$ is a canvas such that $S$ is connected and $\phi$ is an $L$-coloring of $S$ that extends to an $L$-coloring of $G$, then 

$$\log E_T(\phi) \ge (v(T) + 265(3|E(S)|-5|V(S)|))/10000.$$
\end{lemma}
\begin{proof}
We proceed by induction on $v(T)+e(T)$. We may assume that $S$ is induced as otherwise the lemma follows by induction applied to $T/G[S]$.  We may assume that $G$ is connected, as otherwise we may apply induction to the component of $G$ containing $S$ and Theorem~\ref{ThomExp} to the components of $G$ not containing $S$.

We may also assume that $v(T) > 0$ as otherwise the lemma follows as $\phi$ extends to at least one $L$-coloring of $G$.

Let $v\in N(S)\setminus S$. Suppose that $|N(v)\cap V(S)|\ge 2$. Let $H=G[V(S) \cup v]$. Extend $\phi$ to $H$ such that the coloring also extends to $G$. Note that $v(T/H) < v(T)$. Hence by induction, 
$$10000 \log E_{T/H}(\phi) \ge v(T/H)+265(3|E(H)|-5|V(H)|).$$
Yet $E_T(\phi)\ge E_{T/H}(\phi)$. As $v(T/H)=v(T)-1$, $|E(H)|\ge|E(S)|+2$ and $|V(H)|=|V(S)|+1$, we find that 
\begin{align*}
v(T/H)+148(3|E(H)|-5|V(H)|)&\ge v(T)-1 + 265(3(|E(S)|+2)-5(|V(S)|+1))\\ 
&= v(T) + 265 (3|E(S)|-5|V(S)|) + 264,
\end{align*} 
and the result follows.

So we may suppose that $|N(v)\cap V(S)|=1$. Let $H=G[V(S)\cup v]$. Let $S(v)= L(v)\setminus \{\phi(u)| u\in N(v)\cap V(S)\}$. As $|N(v)\cap V(S)|=1$, $|S(v)|\ge 2$. Let $c_1,c_2\in S(v)$. For $i\in\{1,2\}$, let $\phi_i(v)=c_i$ and $\phi_i(u)=\phi(u)$ for all $u\in V(S)$. If both $\phi_1$ and $\phi_2$ extend to $L$-colorings of $G$ (and hence $H$), it follows by induction applied to $T/H$ that 
$$10000 \log E_{T/H}(\phi_i)\ge (v(T/H) + 265(3|E(H)|-5|V(H)|).$$ 
As $v(T/H)=v(T)-1$, $|E(H)|=|E(S)|+1$ and $|V(H)|=|V(S)|+1$ we find that 
\begin{align*}
10000 \log E_{T/H} (\phi_i)&\ge v(T)-1 + 265 (3(|E(S)|+1) - 5 (|V(S)|+1)) \\
&=  v(T)+265(3|E(S)|-5|V(S)|) - 529.
\end{align*}
Yet $E_T(\phi)\ge E_{T/H}(\phi_1)+E_{T/H}(\phi_2)$ and hence 
$$10000 \log E_T(\phi)\ge v(T)+265(3|E(S)|-5|V(S)|),$$ 
and the lemma follows.

So we may suppose without loss of generality that $\phi_1$ does not extend to an $L$-coloring of $G$. Hence there exists a critical subcanvas $T'=(G', H, L)$ of $T/H$. By Theorem~\ref{StrongLinear}, $v(T') \le 88{\rm def}(T')$. By Theorem~\ref{StrongLinear}, it also follows that ${\rm def}(T')\ge 3$. Thus $v(T|G') = v(T')+1$. Moreover, ${\rm def}(T')\le {\rm def}(T|G')+2$ and hence ${\rm def}(T|G')\ge 1$. Hence 
$$v(T|G')\le 1+88{\rm def}(T') \le 1+88({\rm def}(T|G')+2) \le 265{\rm deg}(T|G').$$ 
Also recall that ${\rm def}(T|G')= 3e(T|G')-5v(T|G')$. 

As $\phi$ extends to an $L$-coloring of $G$, $\phi$ extends to an $L$-coloring of $H$. By induction, 
$$10000 \log E_{T/G'}(\phi) \ge v(T/G')+265(3|V(G')|-5|V(G')|).$$ Note that $|V(G')|=v(T|G')+|V(S)|$ and $|E(G')|=e(T|G')+|E(S)|$. Thus 
\begin{align*}
10000 \log E_T(\phi) &\ge v(T)-v(T|G') + 265(3(e(T|G')+|E(S)|)-5(v(T|G')+|V(S)|)) \\
&= v(T)+265(3|E(S)|-5|V(S)|) + 265{\rm deg}(T|G') - v(T|G').
\end{align*} 
Since from above $265{\rm def}(T|G')\ge v(T|G')$, the lemma follows.
\end{proof}

Let us now define a notion of criticality for having exponentially many extensions as follows.

\begin{definition}[Definition 5.18 in~\cite{PostleThomas}]
Let $\epsilon,\alpha>0$. Let $G$ be a graph with rings $\R$ embedded in a surface $\Sigma$ of Euler genus $g$, let $R$ be the total number of ring vertices and let $L$ be a list-assignment of $G$. We say that $G$ is \emph{$(\epsilon,\alpha)$-exponentially-critical with respect to $L$} if $G\ne \bigcup \R$ and for every proper subgraph $G'\subseteq G$ that includes all the rings there exists an $L$-coloring $\phi$ of $\bigcup \R$ such that exist $2^{\epsilon (|V(G')\setminus V(H)|-\alpha(g+R))}$ distinct $L$-colorings of $G'$ extending $\phi$ but there does not exist $2^{\epsilon(|V(G)\setminus V(H)|-\alpha(g+R))}$ distinct $L$-colorings of $G$ extending $\phi$.
\end{definition}

We are now ready to prove that the family of graphs of girth at least five with rings which are $(\epsilon,\alpha)$-exponentially-critical with respect to some $3$-list assignment are hyperbolic as long $\epsilon \in (0, 1/20000)$ and $\alpha\ge 0$. In fact, we can prove the stronger result where we relax the girth condition to the condition that every cycle of four or less is equal to a ring.

\begin{theorem}
\label{ExpHyper}
Let $\epsilon>0$ and $\alpha\ge0$, and 
let $\F$ be the family of embedded graphs $G$ with rings such that 
every cycle of length four or less is equal to a ring and $G$ is $(\epsilon,\alpha)$-exponentially-critical with respect to some $3$-list assignment. If $\epsilon<1/20000$, then the family $\F$ is hyperbolic with Cheeger constant $6908$ (independent of $\epsilon$ and $\alpha$). 
\end{theorem}
\begin{proof}
Let $G$ be a  graph with rings $\R$ embedded in a surface $\Sigma$ of Euler genus $g$ such that every cycle of length four or less is not null-homotopic
and such that $G$ is $(\epsilon,\alpha)$-exponentially-critical with respect to  
a $3$-list assignment $L$, let $R$ be the total number of ring vertices, and
let $\gamma:{\mathbb S}^1\to\Sigma$ be a closed curve that bounds
an open disk $\Delta$ and intersects $G$ only in vertices.
To avoid notational complications we will assume that $\gamma$ is a simple
curve; otherwise we split vertices that $\gamma$ visits more than once to
reduce to this case.
We may assume that $\Delta$ includes at least one vertex of $G$, for 
otherwise there is nothing to show.
Let $X$ be the set of vertices of $G$ intersected by $\gamma$.
Then $|X|\ge 2$ by Theorem~\ref{ThomExp2} and further if $|X|=2$, then $X$ is an independent set.

Let $G_0$ be the subgraph of $G$ consisting of all vertices and edges drawn 
in the closure  of $\Delta$.
Let $G_1$ be obtained from $G_0$ as follows. For every pair of vertices $u,v\in X$
that are consecutive on the boundary of $\Delta$ we do the following. We delete the edge $uv$ if it exists and then we introduce a path of two new degree two vertices joining $u$ and $v$, embedding the new edges and vertices in a segment of $\gamma$.

Thus $G_1$ has a cycle $C_1$ embedded in the image of $\gamma$, and hence
$G_1$ may be regarded as a plane graph with outer cycle $C_1$.
For $v\in V(G_0)$ let $L_1(v):=L(v)$, and for $v\in V(G_1)-V(G_0)$ let $L_1(v)$
be an arbitrary set of size three. 

Let $T=(G_1,C_1,L_1)$. Note that $G_1$ has girth at least five and so $T$ is a canvas. Also note that $|C_1|\le 3|X|$.

We may assume for a contradiction that 
$$|V(G_0)-X|> 6908(|X|-1)\ge 3454|X|,$$
where the last inequality follows since $|X|\ge 2$.

Let $G_2$ be the smallest subgraph of  $G_1$ such that
$G_2$ includes $C_1$  as a subgraph and every $L_1$-coloring of $C_1$ that extends to an $L_1$-coloring
of $G_2$ also extends to an $L_1$-coloring of $G_1$.
Then $G_2$ is $C_1$-critical with respect to $L_1$. Hence $T|G_2$ is a critical canvas or $T|G_2=T|C_1$.  By Corollary~\ref{StrongLinear3}, 
$$|V(G_2)-V(C_1)|\le 88 |V(C_1)|\le 274|X|.$$ 

Let $G_0'=G\setminus (V(G_1)- V(G_2))$. Thus $G_0'$ is a proper subgraph of $G$, and, in fact,
\begin{align*}
|V(G)-V(G_0')|&= |V(G_0)-V(G_0')|\\
&=|V(G_0)-X|-|V(G_2)-V(C_1)|\\
&\ge 3454|X|-274|X| = 3180|X|.
\end{align*}

As $G$ is $(\epsilon,\alpha)$-exponentially-critical with respect to $L$, there exists an
$L$-coloring $\phi$ of $\bigcup\R$ such that $\phi$ extends to at least
 $2^{\epsilon (|V(G_0')|-\alpha(g+R))}$
distinct $L$-colorings of $G_0'$,
but does not extend to at least $2^{\epsilon (|V(G)|-\alpha(g+R))}$ distinct $L$-colorings of $G$.

\begin{claim}\label{Extend}
If $\phi'$ is an $L$-coloring of $G_0'$ that extends $\phi$, then $\phi'$ extends to at least 
$2^{( |V(G)-V(G_0')|)/20000}$ distinct $L$-colorings of $G$.
\end{claim}
\begin{proof}
Note that $\phi'$ extends to an $L_1$-coloring $\phi''$ of $G_2$ by colorings the paths of degree two vertices in $G_1-V(G_0)$. By the definition of $G_2$, every $L_1$-coloring of $G_2$ extends to an $L_1$-coloring of $G_1$. Hence $\phi''$ extends to an $L_1$-coloring of $G_1$. Thus by Lemma~\ref{ExpManyDisc} 

$$10000 \log E_{T/G_2}(\phi'') \ge v(T/G_2) + 265(3|E(G_2)|-5|V(G_2)|).$$

However, every such extension induces a different extension of $\phi'$ to $G$. Note that $v(T/G_2)=|V(G_0)|-|V(G_2)|=|V(G)|-|V(G_0')|$. Moreover, $3|E(G_2)|-5|V(G_2)| = {\rm def}(T|G_2) + (3|E(C_1)|-5|V(C_1)|)$. As $T|G_2$ is either critical or equal to $T|C_1$, it follows from Theorem~\ref{StrongLinear} that ${\rm def}(T|G_2)\ge 0$. Since $|E(C_1)|=|V(C_1)|$ and $|V(C_1)|\le 3|X|$, we find that  
$$3|E(G_2)|-5|V(G_2)|\ge - 6|X|.$$
Hence,
$$10000 \log E_{T/G_2}(\phi'') \ge |V(G)|-|V(G_0')| - 1590|X| \ge (|V(G)|-|V(G_0')|)/2,$$

where the last inequality follows since $|V(G)|-|V(G_0')|\ge 3180|X|$.
\end{proof}

The coloring $\phi$ extends to at least $2^{\epsilon (|V(G_0')|-\alpha(g+R))}$
distinct $L$-colorings of $G_0'$. By Claim~\ref{Extend}, each such extension $\phi'$ extends to at least 
$2^{|V(G)-V(G_0')|/20000}$ distinct $L$-colorings of $G$. But then as $\epsilon\le 1/20000$, there exist at least $2^{\epsilon (|V(G)|-\alpha(g+R))}$ 
distinct $L$-colorings of $G$ extending $\phi$, a contradiction.
\end{proof}

\end{document}